\documentclass[a4paper,11pt]{article}
\usepackage{amsmath,amsthm}
\usepackage{authblk}
\usepackage[style=numeric-comp,maxbibnames=99,bibencoding=utf8,firstinits=true]{biblatex}
\usepackage{booktabs}
\usepackage[makeroom]{cancel}
\usepackage{enumitem}
\usepackage[hmargin={26mm,26mm},vmargin={30mm,35mm}]{geometry}
\usepackage{loglogslopetriangle}
\usepackage{multirow}
\usepackage{nicefrac}
\usepackage[colorlinks,allcolors={blue}]{hyperref}
\usepackage{pgfplots}
\usepackage[compact,tiny]{titlesec}
\usepackage{tikz-cd}
\usepackage{subcaption}
\usepackage{tensor}

\usepackage{pgfplots,pgfplotstable}
\usepackage[font=small]{subcaption}

\usepackage{newtxtext}
\usepackage{newtxmath}
\usetikzlibrary{patterns}

\bibliography{main}
\newcommand{\email}[1]{\href{mailto:#1}{#1}}

\numberwithin{equation}{section}

\newtheorem{theorem}{Theorem}
\newtheorem{proposition}[theorem]{Proposition}

\theoremstyle{remark}
\newtheorem{remark}[theorem]{Remark}
\theoremstyle{definition}

\newcommand{\st}{\,:\,}
\newcommand{\Real}{\mathbb{R}}

\DeclareRobustCommand{\bvec}[1]{\boldsymbol{#1}}
\newcommand{\uvec}[1]{\underline{\bvec{#1}}}
\newcommand{\cvec}[1]{\bvec{\mathcal{#1}}}

\DeclareMathOperator{\GRAD}{\bf grad}
\DeclareMathOperator{\CURL}{\bf curl}
\DeclareMathOperator{\DIV}{div}
\DeclareMathOperator{\ROT}{rot}
\DeclareMathOperator{\VROT}{\bf rot}

\newcommand{\compl}{{\rm c}}
 
\newcommand{\Hcurl}[1]{\bvec{H}(\CURL;#1)}

\newcommand{\Hdiv}[1]{\bvec{H}(\DIV;#1)}

\newcommand{\Xgrad}[2]{\underline{X}_{\GRAD,#2}^{#1}}
\newcommand{\Xcurl}[2]{\uvec{X}_{\CURL,#2}^{#1}}
\newcommand{\Xdiv}[2]{\uvec{X}_{\DIV,#2}^{#1}}
\newcommand{\Xbullet}[2]{\underline{X}_{\bullet,#2}^{#1}}

\newcommand{\Id}[1]{\boldsymbol{\sf I}_{#1}}

\newcommand{\sfP}{{\mathsf{P}}}

\newcommand{\bdry}[1]{\mathcal{B}_{#1}}
\newcommand{\sfb}{{\mathsf{b}}}

\newcommand{\Igrad}[2]{\underline{I}_{\GRAD,#2}^{#1}}
\newcommand{\Icurl}[2]{\uvec{I}_{\CURL,#2}^{#1}}
\newcommand{\Idiv}[2]{\uvec{I}_{\DIV,#2}^{#1}}

\newcommand{\lproj}[2]{\pi_{\Poly{},#2}^{#1}}

\newcommand{\Rproj}[2]{\bvec{\pi}_{\cvec{R},#2}^{#1}}
\newcommand{\Rcproj}[2]{\bvec{\pi}_{\cvec{R},#2}^{\compl,#1}}

\newcommand{\Gproj}[2]{\bvec{\pi}_{\cvec{G},#2}^{#1}}
\newcommand{\Gcproj}[2]{\bvec{\pi}_{\cvec{G},#2}^{\compl,#1}}

\newcommand{\uGh}[1]{\uvec{G}_h^{#1}}

\newcommand{\uCh}[1]{\uvec{C}_h^{#1}}

\newcommand{\Dh}[1]{D_h^{#1}}

\newcommand{\GE}[1]{G_E^{#1}}
\newcommand{\cGF}[1]{\boldsymbol{\mathsf{G}}_F^{#1}}
\newcommand{\cGT}[1]{\boldsymbol{\mathsf{G}}_T^{#1}}

\newcommand{\CF}[1]{C_F^{#1}}
\newcommand{\cCT}[1]{\boldsymbol{\mathsf{C}}_T^{#1}}
\newcommand{\cCh}[1]{\boldsymbol{\mathsf{C}}_h^{#1}}

\newcommand{\DT}[1]{D_T^{#1}}

\newcommand{\trE}[1]{\gamma_E^{#1}}
\newcommand{\trF}[1]{\gamma_F^{#1}}
\newcommand{\trFt}[1]{\bvec{\gamma}_{{\rm t},F}^{#1}}

\newcommand{\faces}[1]{\mathcal{F}_{#1}}
\newcommand{\edges}[1]{\mathcal{E}_{#1}}
\newcommand{\vertices}[1]{\mathcal{V}_{#1}}

\newcommand{\FT}{\faces{T}}
\newcommand{\ET}[1][T]{\edges{#1}}
\newcommand{\EF}{\edges{F}}

\newcommand{\VE}{\vertices{E}}

\newcommand{\normal}{\bvec{n}}
\newcommand{\tangent}{\bvec{t}}

\newcommand{\Poly}[2][]{\mathcal{P}_{#1}^{#2}}
\newcommand{\vPoly}[2][]{\cvec{P}_{#1}^{#2}}
\newcommand{\Roly}[1]{\cvec{R}^{#1}}
\newcommand{\Goly}[1]{\cvec{G}^{#1}}
\newcommand{\cRoly}[1]{\cvec{R}^{\compl,#1}}
\newcommand{\cGoly}[1]{\cvec{G}^{\compl,#1}}

\newcommand{\norm}[2]{\|#2\|_{#1}}
\newcommand{\Mh}[1][h]{\mathcal{M}_{#1}}
\newcommand{\Th}[1][h]{\mathcal{T}_{#1}}
\newcommand{\Fh}[1][h]{\mathcal{F}_{#1}}
\newcommand{\Eh}[1][h]{\mathcal{E}_{#1}}
\newcommand{\Vh}[1][h]{\mathcal{V}_{#1}}

\newcommand{\Pgrad}[1]{P_{\GRAD,T}^{#1}}
\newcommand{\Pcurl}[1]{\bvec{P}_{\CURL,T}^{#1}}
\newcommand{\Pgradh}[1]{P_{\GRAD,h}^{#1}}
\newcommand{\Pcurlh}[1]{\bvec{P}_{\CURL,h}^{#1}}
\newcommand{\Pdiv}[1]{\bvec{P}_{\DIV,T}^{#1}}
\newcommand{\Pdivh}[1]{\bvec{P}_{\DIV,h}^{#1}}
\newcommand{\Pbullet}[1]{P_{\bullet,T}^{#1}}

\newcommand{\SG}[1]{\bvec{S}_{\GRAD,#1}^k}
\newcommand{\SC}[1]{\bvec{S}_{\CURL,#1}^k}

\newcommand{\Hstar}{{\star}} % Hodge star

\newcommand{\La}{\mathfrak{g}}
\newcommand{\LaXgrad}[2]{\Xgrad{#1,\La}{#2}}
\newcommand{\LaXcurl}[2]{\Xcurl{#1,\La}{#2}}
\newcommand{\LaXdiv}[2]{\Xdiv{#1,\La}{#2}}
\newcommand{\LaXbullet}[2]{\Xbullet{#1,\La}{#2}}
\newcommand{\LaHcurl}[1]{\bvec{H}(\CURL;#1)\otimes\La}

\newcommand{\LaHgrad}[1]{H^1(#1)\otimes\La}
\newcommand{\LaIgrad}[2]{\Igrad{#1,\La}{#2}}
\newcommand{\LaIcurl}[2]{\Icurl{#1,\La}{#2}}

\newcommand{\LauGh}[1]{\uGh{#1,\La}}
\newcommand{\LauCh}[1]{\uCh{#1,\La}}

\newcommand{\LaDh}[1]{\Dh{#1,\La}}
\newcommand{\LatrFt}[1]{\trFt{#1,\La}}

\newcommand{\LaPgrad}[1]{\Pgrad{#1,\La}}
\newcommand{\LaPcurl}[1]{\Pcurl{#1,\La}}
\newcommand{\LaPgradh}[1]{\Pgradh{#1,\La}}
\newcommand{\LaPcurlh}[1]{\Pcurlh{#1,\La}}
\newcommand{\LaPdivh}[1]{\Pdivh{#1,\La}}

\newcommand{\ebkt}[3][]{\Hstar[#2,#3]^{#1}}

\newcommand{\ebkttrk}[3][\DIV,k,h]{\ebkt[#1]{#2}{#3}}

\newcommand{\deltat}{\delta\hspace*{-0.15ex}t}
 
\pgfplotsset{select coords between index/.style 2 args={
    x filter/.code={
        \ifnum\coordindex<#1\fi
        \ifnum\coordindex>#2\fi
    }
}}
\begin{document}

\title{Two arbitrary-order constraint-preserving schemes for the Yang--Mills equations on polyhedral meshes}

\author[1,2]{J\'er\^ome Droniou}
\author[1]{Jia Jia Qian}

\affil[1]{School of Mathematics, Monash University, Melbourne, Australia, \email{jerome.droniou@monash.edu}, \email{jia.qian@monash.edu}}
\affil[2]{IMAG, Univ Montpellier, CNRS, Montpellier, France, \email{jerome.droniou@umontpellier.fr}}

\maketitle

\begin{abstract}

Two numerical schemes are proposed and investigated for the Yang--Mills equations, which can be seen as a nonlinear generalisation of the Maxwell equations set on Lie algebra-valued functions, with similarities to certain formulations of General Relativity. Both schemes are built on the Discrete de Rham (DDR) method, and  inherit from its main features: an arbitrary order of accuracy, and applicability to generic polyhedral meshes. They make use of the complex property of the DDR, together with a Lagrange-multiplier approach, to preserve, at the discrete level, a nonlinear constraint associated with the Yang--Mills equations. We also show that the schemes satisfy a discrete energy dissipation (the dissipation coming solely from the implicit time stepping). Issues around the practical implementations of the schemes are discussed; in particular, the assembly of the local contributions in a way that minimises the price we pay in dealing with nonlinear terms, in conjunction with the tensorisation coming from the Lie algebra.
Numerical tests are provided using a manufactured solution, and show that both schemes display a convergence in $L^2$-norm of the potential and electrical fields in $\mathcal O(h^{k+1})$ (provided that the time step is of that order), where $k$ is the polynomial degree chosen for the DDR complex. We also numerically demonstrate the preservation of the constraint.

\end{abstract}

%% \tableofcontents

\section{Introduction}\label{sec:intro}

In this paper we investigate two arbitrary-order numerical methods for the Yang--Mills equations on general polyhedral meshes, based on the fully discrete serendipity Discrete de Rham (SDDR) complex \cite{Di-Pietro.Droniou:23*2}. The first method was proposed (but not tested) in \cite{Droniou.Oliynyk.ea:23} for the non-serendipity version of the Discrete de Rham (DDR) sequence, while the second method is novel to this paper. The two discretisations differ in the treatment of one of the nonlinearities present in these equations. In contrast to conforming methods, the discrete structure of the SDDR spaces means that there is no obvious construction of the nonlinear terms, and this can be problematic when specific algebraic manipulations need to be reproduced, for instance to prove consistency estimates. The implementation cost is another important factor to the viability of each approach, which is explored with accompanying numerical results on the convergence and discrete conservation properties of each scheme.

The classical Yang--Mills equations come from a class of non-abelian gauge theories, generalising the abelian $U(1)$ group of electromagnetism to certain non-abelian gauge groups. Once quantised, this theory forms the foundation of the current Standard Model of particle physics. In the classical setting, the non-commutativity of the group manifests as the appearance of nonlinear quantities in addition to the linear Maxwell terms. Analogous to Maxwell, the Yang--Mills equations can be formulated as a set of evolution equations preserving particular constraints (e.g. the conservation of charge), given that the initial data satisfies these constraints. In the linear case, the preservation of these constraints is a direct consequence of the calculus formula $\DIV\CURL=0$, which is linked to the complex property of the de Rham sequence. Designing numerical methods that replicate this property is essential to maintaining constraint preservation at the discrete level, and thus to obtaining stable schemes. Much work has been done in the Finite Element framework to design discrete versions of the de Rham complex, see, e.g., \cite{Arnold.Falk.ea:06, Arnold:18,Arnold.Falk.ea:10,Gillette.Hu.ea:20,Arnold.Hu:21,Di-Pietro.Hanot:23} and references therein. Finite Element methods are, however, limited to meshes made of specific elements (mostly tetrahedra and hexahedra in 3D), which limits their flexibility in terms of mesh refinement or agglomeration. Recently, discrete polytopal complexes -- discrete versions of continuous complexes, that are applicable on meshes made of generic polyhedra -- have been introduced, see, e.g., \cite{Beirao-da-Veiga.Brezzi.ea:16,Beirao-da-Veiga.Brezzi.ea:18*2,Di-Pietro.Droniou.ea:20,Di-Pietro.Droniou:23*1}. The discrete complex property enabled the design of stable and robust schemes, in particular for magnetostatics \cite{Di-Pietro.Droniou:21*1,Beirao-da-Veiga.Brezzi.ea:18*2}, plate problems \cite{Di-Pietro.Droniou:21,Di-Pietro.Droniou:23,Chen.Huang:18,Chen.Huang:22}, and the Stokes equations \cite{Beirao-da-Veiga.Dassi.ea:22,Beirao-da-Veiga.Dassi.ea:20}). 

Given the importance, for the stability of schemes, of preserving constraints at the discrete level, similar techniques have been explored for the Yang--Mills equations, using either Finite Element or polytopal approaches \cite{Droniou.Oliynyk.ea:23,Christiansen.Winther:06,Berchenko-Kogan.Stern:21}. For these equations, however, the nonlinearity has proven to be troublesome, and required additional techniques (e.g., the introduction of Lagrange multipliers) beyond a discrete version of the formula $\DIV\CURL=0$. The interest in developing our understanding of such methods is in the application to numerical schemes for Einstein's equations, where the absence of this constraint propagation can cause disastrous error growth \cite{Alic.Bona-Casas.ea:12,Brodbeck.Frittelli.ea:99,Frauendiener.Vogel:05} in the numerical simulations. Current techniques to control this error include constraint damping \cite{Alic.Bona-Casas.ea:12,Brodbeck.Frittelli.ea:99,Gundlach.Calabrese.ea:05}, where specific terms are added to the evolution equations to suppress the growth of the constraint violations, but methods for exact preservation remain limited. The link with the Yang--Mills equations is that in certain formulations of General Relativity (GR), such as the Einstein-Bianchi system \cite{Friedrich:96, Anderson.Choquet-Bruhat.ea:97}, these equations can resemble greatly those of electromagnetism - with additional nonlinear terms. Therefore it is natural to expect that these ideas will aid in designing a constraint preserving scheme for GR based on the framework of discrete polytopal complexes.

An equally important aspect of the design of numerical methods is the feasibility of the implementation and testing under real world conditions. We find more commonly, for the Yang--Mills equations, numerical tests run in only low-order 2D settings \cite{Christiansen.Winther:06,Berchenko-Kogan.Stern:21}. Any increase in the dimension or the order of the approximation generally leads to schemes that are vastly more expensive to run, and this is compounded by the nonlinearity of the model. Hence working with spaces that are smaller and more refined is an effective way to cut the cost of the simulations. The serendipity Discrete de Rham complex, introduced in \cite{Di-Pietro.Droniou:23*2}, is a variant of the Discrete de Rham complex \cite{Di-Pietro.Droniou.ea:20,Di-Pietro.Droniou:23*1}, where the spaces have undergone a serendipity reduction, eliminating many unknowns, while retaining the complex and consistency properties of the original sequence. This enables the seamless transfer of any DDR scheme and results to the SDDR version, with all the flexibility of the general-order polytopal method at a lower cost. Additionally, this can be combined with other reduction techniques such as static condensation to further increase the efficiency.

The issue with the nonlinearity in the practical implementation is the computations involving `high dimensional' arrays that are required to deal with all the coefficients. This number grows exponentially with the degree of the multilinearity, and thus takes up majority of the time in the assembly phase of the runtime. For matrices (2-dimensional arrays), there exists many specialised algorithms to speed up calculations, as well as efficient storage structures in the case that it is sparse. The libraries for higher dimensional arrays are less advanced, and often incomplete in their features; as a consequence, operations need to be done manually, introducing another source of possible inefficiencies. Simply rearranging the order of calculations can lead to sizeable differences in the space and time complexities, therefore finding the optimal trade-off is key to measuring the actual performance of the scheme.

The paper is organised as follows. In Section \ref{sec:ddr}, we give a presentation of the SDDR complex and its Lie algebra extension, that is independent of the DDR framework. Section \ref{sec:schemes} starts with the constrained formulation of the continuous Yang--Mills equations. Based on that, we introduce the two schemes that are considered in the paper and the differing approaches on the nonlinear terms. This is followed by a proof of the preservation of a discrete constraint functional, as well as energy estimates. Section \ref{sec:implementation} covers the major steps in the implementation of the scheme, showing the impact of Lie algebra tensorisation on the physical data structures, and also highlighting how the trilinear and quadrilinear forms and sums are managed in the tensorised situation. Numerical results for these implementations are found in Section \ref{sec:tests}, where we test both the convergence and the discrete constraint preservation on three different mesh families in the 3D setting. Expected convergence rates of $k+1$ are mostly seen, as well as the preservation of the initial constraint up to machine precision. We also report on the differences in the results and runtimes, which turned out to be very minor between the two methods. A brief conclusion is provided in Section \ref{sec:conclusion}.

%% DDR
\section{Lie Algebra-valued serendipity Discrete De Rham complex}\label{sec:ddr}

We present here the serendipity version of the arbitrary-order Lie Algebra-valued DDR complex, originally sketched in \cite[Section 6]{Droniou.Oliynyk.ea:23}. This complex consists in tensorising the SDDR complex of \cite{Di-Pietro.Droniou:23*2}, which is built in this reference from the (regular) DDR complex; such a presentation relies on a complete description of the latter complex, together with  ``extension'' and ``reduction'' maps that link the two complexes. In the following, we adopt a stand-alone description of the SDDR complex, directly translating the formulas resulting from the links with the DDR complex. For this reason, the notations adopted below differ slightly from \cite{Di-Pietro.Droniou:23*2}: there, the serendipity spaces and operators are denoted using a hat (the non-hat version referring to the regular DDR spaces and operators, which are not needed here).

\subsection{Mesh notations}

We use the same mesh and polynomial space notations as in \cite{Di-Pietro.Droniou:23*1}. Let $U$ be a polygonal domain of $\Real^3$. A mesh $\Mh=\Th\cup\Fh\cup\Eh\cup\Vh$ is a collection of  polyhedral elements (gathered in $\Th$, and partitioning $U$), of polygonal faces (gathered in $\Fh$), of edges (gathered in $\Eh$) and vertices (gathered in $\Vh$). Each $\sfP\in \Mh$ is assumed to be topologically trivial (simply connected with connected boundary), and we denote by $h_{\sfP}$ the diameter of $\sfP$; we set $h=\max_{T\in\Th}h_T$. When applicable, the sets $\Fh[\sfP]$ (resp.~$\Eh[\sfP]$, resp.~$\Vh[\sfP]$) gather the faces (resp.~edges, resp.~vertices) of $\sfP$. Each face $F\in\Fh$ is oriented by the choice of a unit normal $\normal_F$, and each edge $E\in\Eh$ is oriented by the choice of a unit tangent $\tangent_E$. If $T\in\Th$ and $F\in\FT$, $\omega_{TF}$ is the relative orientation of $F$ with respect to $T$: $\omega_{TF}=+1$ if $\normal_F$ points outside $T$, $\omega_{TF}=-1$ otherwise. For $F\in\Fh$ and $E\in\EF$, $\omega_{FE}$ denotes the relative orientation of $E$ with respect to $F$: $\omega_{FE}=+1$ if, along $\partial F$, $\tangent_E$ points counter-clockwise with respect to the orientation of $F$ induced by $\normal_F$, and $\omega_{FE}=-1$ otherwise; we also denote by $\normal_{FE}$ the unit vector such that $(\tangent_E, \normal_{FE},\normal_F)$ defines a right-handed system in $\Real^3$. The final orientation is that of the vertices of each edge: for $E\in\Eh$ and $V\in\VE$, $\omega_{EV}=+1$ if $\tangent_E$ points towards $V$ on $E$, and $\omega_{EV}=-1$ otherwise.

We assume that $\Th\cup\Fh$ satisfies the regularity assumption of \cite[Definition 1.9]{Di-Pietro.Droniou:20} with regularity parameter $\varrho$, and we write $A\lesssim B$ when $A\le CB$ for some $C$ depending only on $U$, $\varrho$ and the possible polynomial degrees involved in $A,B$.

For any mesh face $F\in\Fh$ and smooth enough function $r:F\to\Real$, $\GRAD_F r$ is the gradient of $r$ on $F$ and $\VROT_F r$ its 2-dimensional vector curl (rotation of $\GRAD_F r$ by $-\pi/2$ in the plane spanned by $F$). For a smooth function $\bvec{z}$ on $F$ with values in the tangent plane of $F$, the divergence of $\bvec{z}$ on $F$ is $\DIV_F \bvec{z}$, and its scalar curl (divergence of the rotated by $-\pi/2$ of $\bvec{z}$) is $\ROT_F\bvec{z}$.

If $\sfP\in\Mh$ and $\ell\ge 0$ is an integer, $\Poly{\ell}(\sfP)$ denotes the space of restrictions to $\sfP$ of three-variate polynomials on $\Real^3$ of total degree $\le\ell$, and  $\Poly{0,\ell}(\sfP)$ is its subspace of polynomials with vanishing integral over $\sfP$. We adopt the convention $\Poly{\ell}(\sfP)=\{0\}$ if $\ell<0$. If $T\in\Th$, we set $\vPoly{\ell}(T)=\Poly{\ell}(T)^3$ and, for $F\in\Fh$, $\vPoly{\ell}(F)$ is the subspace of $\Poly{\ell}(F)^3$ which take value in the tangent space of $F$.
The $L^2$-orthogonal projector on $\Poly{\ell}(\sfP)$ is denoted by $\lproj{\ell}{\sfP}$. Selecting, for each $\sfP\in\Th\cup\Fh$, a point $\bvec{x}_\sfP\in \sfP$ such that $\sfP$ contains a ball centered at $\bvec{x}_\sfP$ and of radius $\gtrsim h_\sfP$, we recall the following decompositions of vector-valued polynomial spaces: For all $F\in\Fh$,
\[ 
\vPoly{\ell}(F)=\Roly{\ell}(F)\oplus \cRoly{\ell}(F)\quad\mbox{ with }\quad \Roly{\ell}(F)=\VROT_F\Poly{\ell+1}(F)\mbox{ and }\cRoly{\ell}(F)=(\bvec{x}-\bvec{x}_F)\Poly{\ell-1}(F)
\]
and, for $T\in\Th$,
\begin{align*}
\vPoly{\ell}(T)={}&\Roly{\ell}(T)\oplus \cRoly{\ell}(T)\quad\mbox{ with }\quad \Roly{\ell}(T)=\CURL\Poly{\ell+1}(T)\mbox{ and }\cRoly{\ell}(T)=(\bvec{x}-\bvec{x}_T)\Poly{\ell-1}(T),\\
\vPoly{\ell}(T)={}&\Goly{\ell}(T)\oplus \cGoly{\ell}(T)\quad\mbox{ with }\quad \Goly{\ell}(T)=\GRAD\Poly{\ell+1}(T)\mbox{ and }\cGoly{\ell}(T)=(\bvec{x}-\bvec{x}_T)\times\vPoly{\ell-1}(T)
\end{align*}
(here and in the following, when used between two vectors or a vector and a space, $\times$ denotes the cross product in $\Real^3$).
The $L^2$-orthogonal projectors on these spaces are, with obvious notations, $\Rproj{\ell}{F}$, $\Rcproj{\ell}{F}$, $\Rproj{\ell}{T}$, $\Rcproj{\ell}{T}$, $\Gproj{\ell}{T}$ and $\Gcproj{\ell}{T}$.

\subsection{Serendipity DDR complex}\label{sec:SDDR.complex}

For each element or face $\sfP\in\Th\cup\Fh$, we select on the boundary of the element (resp.~face) a set $\bdry{\sfP}$ of $\eta_\sfP\ge 2$ faces (resp.~edges) that are not pairwise coplanar (resp.~aligned) and such that, for each $\sfb\in\bdry{\sfP}$, $\sfP$ lies entirely on one side of the affine space spanned by $\sfb$. From here on, we fix a polynomial degree $k\ge 0$, measuring the accuracy of the discrete complex, and we set
\[
\ell_\sfP=k+1-\eta_\sfP\qquad\forall \sfP\in\Th\cup\Fh.
\]

\subsubsection{Spaces and serendipity operators}

The SDDR versions of the $H^1(U)$, $\Hcurl{U}$, $\Hdiv{U}$ and $L^2(U)$ spaces appearing in the continuous de Rham complex are the following spaces.
\begin{align*}
  \Xgrad{k}{h}&\coloneq
  \Big\{
    \begin{aligned}[t]
      \underline{q}_h&=((q_T)_{T\in\Th},(q_F)_{F\in\Fh},(q_E)_{E\in\Eh},(q_V)_{V\in\Vh})\st
      \\
      {}&
       \text{$q_T\in\Poly{\ell_T}(T)$ for all $T\in\Th$, $q_F\in\Poly{\ell_F}(F)$ for all $F\in\Fh$},\\
      & \text{$q_E\in\Poly{k-1}(E)$ for all $E\in\Eh$, and $q_V\in\Real$ for all $V\in\Vh$}
    \Big\},
    \end{aligned}\\
  \Xcurl{k}{h}&\coloneq
  \Big\{
  \begin{aligned}[t]
    \uvec{v}_h&=((\bvec{v}_{\cvec{R},T},\bvec{v}_{\cvec{R},T}^\compl)_{T\in\Th},(\bvec{v}_{\cvec{R},F},\bvec{v}_{\cvec{R},F}^\compl)_{F\in\Fh},(v_E)_{E\in\Eh})
    \st
    \\
    &\qquad
    \text{$\bvec{v}_{\cvec{R},T}\in\Roly{k-1}(T)$ and $\bvec{v}_{\cvec{R},T}^\compl\in\cRoly{\ell_T+1}(T)$ for all $T\in\Th$,}\\
    &\qquad
    \text{$\bvec{v}_{\cvec{R},F}\in\Roly{k-1}(F)$ and $\bvec{v}_{\cvec{R},F}^\compl\in\cRoly{\ell_F+1}(F)$ for all $F\in\Fh$,}\\
    &\qquad
    \text{and $v_E\in\Poly{k}(E)$ for all $E\in\Eh$}      
    \Big\},
  \end{aligned}\\
  \Xdiv{k}{h}&\coloneq
  \Big\{
  \begin{aligned}[t]
    \uvec{w}_h&=((\bvec{w}_{\cvec{G},T},\bvec{w}_{\cvec{G},T}^\compl)_{T\in\Th},(w_F)_{F\in\Fh})
    \st
    \\
    &\qquad
    \text{$\bvec{w}_{\cvec{G},T}\in\Goly{k-1}(T)$ and $\bvec{w}_{\cvec{G},T}^\compl\in\cGoly{k}(T)$ for all $T\in\Th$,}\\
    &\qquad
    \text{and $v_F\in\Poly{k}(F)$ for all $F\in\Fh$}      
    \Big\},
  \end{aligned}\\
  \Poly{k}(\Th)&\coloneq
  \Big\{
  \begin{aligned}[t]
    r_h\in L^2(U)\st \text{$(r_h)_{|T}\in\Poly{k}(T)$ for all $T\in\Th$}
    \Big\}.
  \end{aligned}
\end{align*}
The interpolators on these spaces consist in projecting continuous scalar/vector fields (or some of their traces) onto the polynomial components of the spaces. Specifically, $\Igrad{k}{h}:C(\overline{U})\to\Xgrad{k}{h}$, $\Icurl{k}{h}:\bvec{C}(\overline{U})\to \Xcurl{k}{h}$ and $\Idiv{k}{h}:\bvec{C}(\overline{U})\to \Xdiv{k}{h}$ are defined as:
\begin{alignat*}{4}
\Igrad{k}{h}q={}&((\lproj{\ell_T}{T}q)_{T\in\Th},(\lproj{\ell_F}{F}q)_{F\in\Fh},(\lproj{k-1}{E}q)_{E\in\Eh},(q(\bvec{x}_V))_{V\in\Vh})&\quad\forall q\in C(\overline{U}),\\
\Icurl{k}{h}\bvec{v}={}&( (\Rproj{k-1}{T}\bvec{v},\Rcproj{\ell_T+1}{T}\bvec{v})_{T\in\Th},(\Rproj{k-1}{F}\bvec{v}_{{\rm t},F},\Rcproj{\ell_F+1}{F}\bvec{v}_{{\rm t},F})_{F\in\Fh},(\lproj{k}{E}(\bvec{v}\cdot\tangent_E))_{E\in\Eh})&\quad\forall\bvec{v}\in \bvec{C}(\overline{U}),\\
\Idiv{k}{h}\bvec{w}={}&( (\Gproj{k-1}{T}\bvec{w},\Gcproj{k}{T}\bvec{w})_{T\in\Th},(\lproj{k}{F}(\bvec{w}\cdot\normal_E))_{E\in\Eh})&\quad\forall\bvec{w}\in \bvec{C}(\overline{U}),
\end{alignat*}
where $\bvec{v}_{{\rm t},F}=\normal_F\times(\bvec{v}_{|F}\times\normal_F)$ is the tangential trace of $\bvec{v}$ on $F$.

As usual in fully discrete complexes, we adopt the underlined notation for vectors of polynomial components, and we replace the index $h$ with $\sfP$ to denote the restriction of these spaces (and the operators defined on them) to a mesh entity $\sfP\in\Th\cup\Fh\cup\Eh$ and its boundary entities. So, for example, a vector $\uvec{v}_F\in\Xcurl{k}{F}$ corresponds to $\uvec{v}_F=(\bvec{v}_{\cvec{R},F},\bvec{v}_{\cvec{R},F}^\compl,(v_E)_{E\in\EF})$.

\medskip

The DDR spaces correspond to the spaces above with the choice $\ell_F=\ell_T=k-1$ (that is, $\eta_F=\eta_P=2$). This implies in particular that, for $k=0$ (which forces $\ell_\sfP<0$ for all $\sfP\in\Th\cup\Fh$), the standard and serendipity DDR spaces are identical. However, as soon as $k\ge 1$, the SDDR spaces have lower dimensions, while still encoding the same level of polynomial consistency as the DDR spaces. This is due to the existence of two families of key operators, the serendipity gradient and curl operators.
Specifically, for $\sfP\in\Th\cup\Fh$, the role of the gradient serendipity operator $\SG{\sfP}:\Xgrad{k}{\sfP}\to\vPoly{k}(\sfP)$ is to reconstruct a consistent gradient, while the curl serendipity operator $\SC{\sfP}:\Xcurl{k}{\sfP}\to\vPoly{k}(\sfP)$ reconstructs a consistent vector potential. The consistencies in questions are expressed by the following relations (see \cite[Proposition 18]{Di-Pietro.Droniou:23*2}):
\[
\SG{\sfP}\Igrad{k}{\sfP}q=\GRAD_\sfP q\qquad\forall q\in\Poly{k+1}(\sfP)\,,\quad
\SC{\sfP}\Icurl{k}{\sfP}\bvec{v}=\bvec{v}\qquad\forall \bvec{v}\in\vPoly{k}(\sfP).
\]
We do not present the precise definitions of these operators, which are not essential to describe the SDDR complex, and refer the reader to \cite{Di-Pietro.Droniou:23*2}.

\medskip

In the next three sections we define operators acting on these spaces, with values in full polynomial spaces, mimicking the gradient, curl, and divergence. It should be noted that, in the original presentation of the SDDR complex in \cite{Di-Pietro.Droniou:23*2}, these operators were not explicitly defined -- only the discrete operators (projections on the complex spaces, see Section \ref{sec:sddr.complex}) and discrete inner products were detailed, based on those of the DDR complex. The polynomial operators below correspond to those of the DDR complex composed with the extension operators linking the DDR and SDDR complex; for ease of reference, we indicate which formulas from \cite{Di-Pietro.Droniou:23*2} yield the definitions presented here.

\subsubsection{Operators on the gradient space}

For each edge $E\in\Eh$ we define the edge gradient $\GE{k}:\Xgrad{k}{E}\to\Poly{k}(E)$ and potential reconstruction $\trE{k+1}:\Xgrad{k}{E}\to\Poly{k+1}(E)$ by: For all $\underline{q}_E=(q_E,(q_V)_{V\in\VE})\in\Xgrad{k}{E}$,
\begin{align*}
\int_E \GE{k}\underline{q}_E\,r_E=-\int_E q_E r_E'+\sum_{V\in\VE}\omega_{EV}q_Vr_E(\bvec{x}_V)\qquad\forall r_E\in\Poly{k}(E),\\
\trE{k+1}\underline{q}_E(\bvec{x}_V)=q_V\quad\forall V\in\VE\quad\mbox{ and }\quad\lproj{k-1}{E}(\trE{k+1}\underline{q}_E)=q_E.
\end{align*}
The definition of $\GE{k}\underline{q}_E$, in which the derivative $r_E'$ is taken in the direction $\tangent_E$, mimics an integration-by-parts formula; it can be checked that $\GE{k}\underline{q}_E=(\trE{k+1}\underline{q}_E)'$.

For each $F\in\Fh$, combining \cite[Eqs.~(4.2), (5.18) and (6.6)]{Di-Pietro.Droniou:23*2} together with $\DIV_F\VROT_F=0$ yields the following definition of the face gradient $\cGF{k}:\Xgrad{k}{F}\to\vPoly{k}(F)$: For all $\underline{q}_F\in\Xgrad{k}{F}$,
\[
\int_F\cGF{k}\underline{q}_F\cdot(\bvec{w}+\bvec{\tau})=\sum_{E\in\EF}\omega_{FE}\int_E \trE{k+1}\underline{q}_E (\bvec{w}\cdot\normal_{FE})
+\int_F \SG{F}\underline{q}_F\cdot\bvec{\tau}\qquad\forall (\bvec{w},\bvec{\tau})\in\Roly{k}(F)\times\cRoly{k}(F).
\]
Using this face gradient and \cite[Eqs.~(4.3) and (5.18)]{Di-Pietro.Droniou:23*2}, the scalar potential reconstruction on $F\in\Fh$ is then $\trF{k+1}:\Xgrad{k}{F}\to\Poly{k+1}(F)$ defined by: For all $\underline{q}_F\in\Xgrad{k}{F}$,
\[
\int_F \trF{k+1}\underline{q}_F\DIV_F \bvec{w}=-\int_F \cGF{k}\underline{q}_F\cdot\bvec{w}+\sum_{E\in\EF}\omega_{FE}\int_E \trE{k+1}\underline{q}_E (\bvec{w}\cdot\normal_{FE})\qquad\forall\bvec{w}\in\cRoly{k+2}(F).
\]

Finally, for $T\in\Th$, we use \cite[Eq.~(4.4), (5.32) and (6.6)]{Di-Pietro.Droniou:23*2} to write the element gradient $\cGT{k}:\Xgrad{k}{T}\to\vPoly{k}(T)$ as: For all $\underline{q}_T\in\Xgrad{k}{T}$,
\[
\int_T\cGT{k}\underline{q}_T\cdot(\bvec{w}+\bvec{\tau})=\sum_{F\in\FT}\omega_{TF}\int_F \trF{k+1}\underline{q}_F(\bvec{w}\cdot\normal_{TF})
+\int_T\SG{T}\underline{q}_T\cdot\bvec{\tau}\quad\forall(\bvec{w},\bvec{\tau})\in\Roly{k}(T)\times\cRoly{k}(T).
\]
The potential reconstruction $\Pgrad{k+1}:\Xgrad{k}{T}\to\Poly{k+1}(T)$ is such that: For all $\underline{q}_T\in\Xgrad{k}{T}$,
\[
\int_T \Pgrad{k+1}\underline{q}_T\DIV \bvec{w}=-\int_T \cGT{k}\underline{q}_T\cdot\bvec{w}+\sum_{F\in\FT}\omega_{TF}\int_F \trF{k+1}\underline{q}_F (\bvec{w}\cdot\normal_{TF})\qquad\forall\bvec{w}\in\cRoly{k+2}(T).
\]

\begin{remark}[Approximation properties of the potential reconstructions in $\Xgrad{k}{T}$]
As demonstrated by \cite[Theorem 6]{Di-Pietro.Droniou:23*1}, the potential reconstructions on the space $\Xgrad{k}{T}$ have optimal approximation properties of degree $k+1$. This is however an exception to the rule of spaces and potential reconstructions in the DDR complex; the reasons for this exception are better understood when translating this complex in the language of differential forms (see \cite{Bonaldi.Di-Pietro.ea:23}, especially Remarks 7 and 18 therein).
\end{remark}

\subsubsection{Operators on the curl space}

For $F\in\Fh$, using \cite[Eqs.~(4.6), (5.19) and (6.7)]{Di-Pietro.Droniou:23*2} we define the face curl $\CF{k}:\Xcurl{k}{F}\to\Poly{k}(F)$ by: For all $\uvec{v}_F\in\Xcurl{k}{F}$,
\[
\int_F \CF{k}\uvec{v}_F r=\int_F \bvec{v}_{\cvec{R},F}\VROT_F r - \sum_{E\in\EF}\omega_{FE}\int_E v_E r\qquad\forall r\in\Poly{k}(F).
\]
This definition is actually identical to the face curl in the DDR complex and does not invoke $\SC{F}$, as the extension operators between the SDDR and DDR curl face space do not modify the components on $\Roly{k-1}(F)$ and on $\bigtimes_{E\in\EF}\Poly{k}(E)$. The curl serendipity operator is however involved in the definition of the component on $\cRoly{k}(F)$ of the extension (see \cite[Eq.~(5.19)]{Di-Pietro.Droniou:23*2}), and therefore in the vector potential reconstruction $\trFt{k}:\Xcurl{k}{F}\to\vPoly{k}(F)$ on $F$, defined as: For all $\uvec{v}_F\in\Xcurl{k}{F}$,
\begin{align*}
\int_F \trFt{k}\uvec{v}_F\cdot(\VROT_F r+\bvec{\tau})=\int_F \CF{k}\uvec{v}_F r + \sum_{E\in\EF}\omega_{FE}\int_E v_E r
+{}&\int_F \SC{F}\uvec{v}_F\cdot\bvec{\tau}\\
&\forall (r,\bvec{\tau})\in \Poly{0,k+1}(F)\times\cRoly{k}(F).
\end{align*}

Similar considerations apply to the element curl and vector potential.
For $T\in\Th$, the element curl $\cCT{k}:\Xcurl{k}{T}\to\vPoly{k}(T)$ is defined by: For all $\uvec{v}_T\in\Xcurl{k}{T}$,
\[
\int_T \cCT{k}\uvec{v}_T\cdot\bvec{w}=\int_T \bvec{v}_{\cvec{R},T}\cdot \CURL\bvec{w}+\sum_{F\in\FT}\omega_{TF}\int_F \trFt{k}\uvec{v}_F\cdot(\bvec{w}\times\normal_F)\qquad\forall \bvec{w}\in\vPoly{k}(T).
\]
The vector potential $\Pcurl{k}:\Xcurl{k}{T}\to\vPoly{k}(T)$ is given by: For all $\uvec{v}_T\in\Xcurl{k}{T}$,
\begin{align*}
\int_T \Pcurl{k}\uvec{v}_T\cdot(\CURL\bvec{w}+\bvec{\tau})=\int_T \cCT{k}\uvec{v}_T\cdot\bvec{w}-\sum_{F\in\FT}\omega_{TF}\int_F\trFt{k}{}&\uvec{v}_F\cdot(\bvec{w}\times\normal_F)+\int_T \SC{T}\uvec{v}_T\cdot\bvec{\tau}\\
&\forall (\bvec{w},\bvec{\tau})\in\cGoly{k+1}(T)\times\cRoly{k}(T).
\end{align*}

\subsubsection{Operators on the divergence space}

The discrete divergence and potential on $\Xdiv{k}{T}$, for $T\in\Th$, are identical to those of the DDR complex since no serendipity reduction is actually possible on this space (see \cite[Section 6.5]{Di-Pietro.Droniou:23*2}): $\DT{k}:\Xdiv{k}{T}\to\Poly{k}(T)$ and $\Pdiv{k}:\Xdiv{k}{T}\to\vPoly{k}(T)$ are such that, for all $\uvec{w}_T\in\Xdiv{k}{T}$,
\[
\int_T \DT{k}\uvec{w}_T q=-\int_T\bvec{w}_{\cvec{G},T}\cdot\GRAD q+\sum_{F\in\FT}\omega_{TF}\int_F w_F q\qquad\forall q\in\Poly{k}(T),
\]
\begin{align*}
\int_T \Pdiv{k}\uvec{w}_T\cdot(\GRAD r+\bvec{\tau})=-\int_T \DT{k}\uvec{w}_T r+\sum_{F\in\FT}\omega_{TF}\int_F w_F r{}&+\int_T \bvec{w}_{\cvec{G},T}^\compl \cdot\bvec{\tau}\\
&\forall (r,\bvec{\tau})\in\Poly{0,k+1}(T)\times \cGoly{k}(T).
\end{align*}

\subsubsection{Serendipity DDR complex}\label{sec:sddr.complex}

The serendipity DDR complex is
\[
  \begin{tikzcd}
    \Real\arrow{r}{\Igrad{k}{h}} & \Xgrad{k}{h}\arrow{r}{\uGh{k}} & \Xcurl{k}{h}\arrow{r}{\uCh{k}} & \Xdiv{k}{h}\arrow{r}{\Dh{k}} & \Poly{k}(\Th)\arrow{r}{0} & \{0\},
  \end{tikzcd}
\]
where the discrete differential operators $\uGh{k}$, $\uCh{k}$ and $\Dh{k}$ are obtained projecting the edge/face/element operators onto the proper spaces (dictated by the co-domains):
\begin{align*}
\uGh{k}\underline{q}_h\coloneq{}&\big( (\Rproj{k-1}{T}\cGT{k}\underline{q}_T,\Rcproj{\ell_T+1}{T}\cGT{k}\underline{q}_T)_{T\in\Th},
                                          (\Rproj{k-1}{F}\cGF{k}\underline{q}_F,\Rcproj{\ell_F+1}{F}\cGF{k}\underline{q}_F)_{F\in\Fh},
                                           (\GE{k}\underline{q}_E)_{E\in\Eh}\big),\\
\uCh{k}\uvec{v}_h\coloneq{}&\big( (\Gproj{k-1}{T}\cCT{k}\uvec{v}_T,\Gcproj{k}{T}\cCT{k}\uvec{v}_T)_{T\in\Th},
                                  (\CF{k}\uvec{v}_F)_{F\in\Fh}\big),\\
\Dh{k}\uvec{w}_h\coloneq{}&\big(\DT{k}\uvec{w}_T\big)_{T\in\Th}.
\end{align*}
It was proved that this sequence is indeed a complex \cite{Di-Pietro.Droniou.ea:20,Di-Pietro.Droniou:23*1}, and has the same cohomology as the de Rham complex \cite{Di-Pietro.Droniou.ea:23}.

\subsubsection{Discrete $L^2$-inner products}\label{sec:ddr.l2prod}

To design numerical schemes based on the SDDR complex, an essential ingredient, besides the discrete differential operators, are consistent $L^2$-inner products on the spaces of the complex. A scheme can then be designed by replacing, in the weak formulation of the PDE, the continuous differential operators and $L^2$-products by the discrete operators of the complex and the $L^2$-inner products on its spaces.

The design of these discrete $L^2$-inner products rely on the element potential reconstructions defined in the previous sections. Specifically, if $\Xbullet{k}{h}$ is one of the space $\Xgrad{k}{h}$, $\Xcurl{k}{h}$ or $\Xdiv{k}{h}$ and $\Pbullet{k}$ is the associated potential in the element $T$, the discrete $L^2$-product on $\Xbullet{k}{h}$ is defined by
\[
  (\underline{x}_h,\underline{y}_h)_{\bullet,h}\coloneq\sum_{T\in\Th}(\underline{x}_T,\underline{y}_T)_{\bullet,T}\quad\mbox{with}\quad
  (\underline{x}_T,\underline{y}_T)_{\bullet,T}\coloneq\left[\int_T \Pbullet{k}\underline{x}_T\cdot\Pbullet{k}\underline{y}_T+\mathrm{s}_{\bullet,T}(\underline{x}_T,\underline{y}_T)\right],
\]
where the dot product in the integral is replaced by a multiplication if $\bullet=\GRAD$, and the stabilisation term $\mathrm{s}_{\bullet,T}$ penalises the difference between traces of the element potential and potential reconstructions on the face/edges (where relevant). The precise definition of the stabilisation term therefore depends on the space, and the available traces:
\begin{align}
  \mathrm{s}_{\GRAD,T}(\underline{r}_T,\underline{q}_T)
  \coloneq{}
    &\sum_{F\in\FT}h_F\int_F\big(\Pgrad{k+1}\underline{r}_T-\trF{k+1}\underline{r}_F\big) \big(\Pgrad{k+1}\underline{q}_T-\trF{k+1}\underline{q}_F\big)
    \nonumber\\
    & + \sum_{E\in\ET}h_E^2\int_E \big(\Pgrad{k+1}\underline{r}_T-\trE{k+1}\underline{r}_E\big) \big(\Pgrad{k+1}\underline{q}_T-\trE{k+1}\underline{q}_E\big)\quad\forall \underline{r}_T,\underline{q}_T\in\Xgrad{k}{T},\nonumber\\
   \mathrm{s}_{\CURL,T}(\uvec{w}_T,\uvec{v}_T)\coloneq
   {}&
     \sum_{F\in\FT}h_F\int_F \big( (\Pcurl{k}\uvec{w}_T)_{{\rm t},F}-\trFt{k}\uvec{w}_F\big)\cdot\big( (\Pcurl{k}\uvec{v}_T)_{{\rm t},F}-\trFt{k}\uvec{v}_F\big)
     \nonumber\\
     & + \sum_{E\in\ET}h_E^2\int_E\big(\Pcurl{k}\uvec{w}_T\cdot\tangent_E-w_E\big)\big(\Pcurl{k}\uvec{v}_T\cdot\tangent_E-v_E\big)\quad\forall\uvec{w}_T,\uvec{v}_T\in\Xcurl{k}{T},\nonumber\\
  \mathrm{s}_{\DIV,T}(\uvec{w}_T,\uvec{v}_T)
  \coloneq{}&\sum_{F\in\FT}h_F\int_F\big(\Pdiv{k}\uvec{w}_T\cdot\normal_F-w_F\big)\big(\Pdiv{k}\uvec{v}_T\cdot\normal_F-v_F\big)\quad\forall\uvec{w}_T,\uvec{v}_T\in\Xdiv{k}{T}.
	\label{eq:def.sdiv}
\end{align}

An important property of the potential reconstruction on each mesh entity is their polynomial consistency: applied to interpolates of polynomials of the correct degree $\ell$ ($\ell=k+1$ for the gradient space, $\ell=k$ for the curl and divergence spaces), they return the polynomial itself. This translates into the following polynomial consistency of the $L^2$-inner products:
\[
  (\underline{I}_{\bullet,T}^kf,\underline{I}_{\bullet,T}^kg)_{\bullet,T}=\int_T f\cdot g\qquad\forall f,g\in \Poly{\ell}(T).
\]

\subsection{Lie algebra-valued serendipity DDR complex}

Since the Yang--Mills equations involve Lie algebra-valued functions, a Lie algebra-valued complex is required to discretise them. This complex is simply obtained by tensorisation of the real-valued complex, as in \cite{Droniou.Oliynyk.ea:23}: the spaces are made of Lie algebra-valued polynomials, and the operators of the complex act component by component on the Lie algebra. 

In the following, we consider a Lie algebra $\La$, that is, a finite-dimensional vector space endowed with a bilinear bracket $[\cdot,\cdot]:\La\times\La\to\La$ and an inner product $\langle\cdot,\cdot\rangle:\La\times\La\to\Real$ which satisfy the Jacobi identity
$$
[a,[b,c]]+[b,[c,a]]+[c,[a,b]]=0\quad\forall a,b,c\in\La
$$
and the Ad-invariance property, which implies
$$
\langle [a,b],c\rangle=\langle a,[b,c]\rangle\quad\forall a,b,c\in\La.
$$

We denote the Lie algebra-valued SDDR spaces by appending an exponent $\La$ after the degree $k$. So, for example, the gradient space in the LASDDR (Lie algebra SDDR) complex is
\[
  \LaXgrad{k}{h}\coloneq(\Xgrad{k}{h}\otimes\La)\equiv
  \Big\{
    \begin{aligned}[t]
      \underline{q}_h&=((q_T)_{T\in\Th},(q_F)_{F\in\Fh},(q_E)_{E\in\Eh},(q_V)_{V\in\Vh})\st
      \\
      {}&
       \text{$q_T\in\Poly{\ell_T}(T)\otimes \La$ for all $T\in\Th$, $q_F\in\Poly{\ell_F}(F)\otimes \La$ for all $F\in\Fh$},\\
      & \text{$q_E\in\Poly{k-1}(E)\otimes \La$ for all $E\in\Eh$, and $q_V\in\Real\otimes\La$ for all $V\in\Vh$}
    \Big\}.
    \end{aligned}
\]
We note that, selecting a basis $(e_I)_I$ of $\La$, for any $\sfP\in\Mh$ we have
\[
  \Poly{\ell}(\sfP)\otimes \La\equiv \Poly{\ell}(\sfP;\La)\coloneq \{\phi^Ie_I\,:\,\phi^I\in\Poly{\ell}(\sfP)\}.
\]
Here and in the following we use the implicit summation convention so, for example, $\phi^Ie_I=\sum_I \phi^Ie_I$.
For a general space $X$, an element $v\in X\otimes \La$ can be uniquely decomposed as $v=v^I\otimes e_I$. Any linear operator $L:X\to Y$ acting between two SDDR spaces $X,Y$ (or an SDDR space and a polynomial space) -- such as a discrete differential operator, a potential reconstruction, etc. -- then gives rise to the corresponding LASDDR operator $L^\La:X\otimes \La\to Y\otimes \La$ defined as $L^\La(v)=(L(v^I))\otimes e_I\in Y\otimes \La$; this definition is independent of the choice of the basis in $\La$. With these notations, the LASDDR complex is
\begin{equation*}
  \begin{tikzcd}
    \Real\otimes\La\arrow{r}{\LaIgrad{k}{h}} & \LaXgrad{k}{h}\arrow{r}{\LauGh{k}} & \LaXcurl{k}{h}\arrow{r}{\LauCh{k}} & \LaXdiv{k}{h}\arrow{r}{\LaDh{k}} & \Poly{k}(\Th)\otimes\La\arrow{r}{0} & \{0\}.
  \end{tikzcd}
\end{equation*}
In each space an inner product is obtained by tensorising the inner product of the corresponding SDDR space and of the Lie algebra. So, if $\bullet\in\{\GRAD,\CURL,\DIV\}$,
\[
(\underline{x}_h,\underline{y}_h)_{\bullet,\La,h}=(\underline{x}_h^I,\underline{y}_h^J)_{\bullet,h}\langle e_I,e_J\rangle
\qquad\forall \underline{x}_h=\underline{x}_h^I\otimes e_I\in\LaXbullet{k}{h}\,,\quad
\forall\underline{y}_h=\underline{y}_h^J\otimes e_J\in\LaXbullet{k}{h}.
\]
Practical implementations of the LASDDR complex and related schemes can be easily done, in principle, by tensorising the operators and inner products of an SDDR implementation. Early tensorisation can however lead to unduly expensive calculations, especially when nonlinear terms are involved. We discuss in Section \ref{sec:implementation} the main considerations that must be taken into account to limit the assembly cost in implementations of LASDDR-based schemes. 

%% Schemes
\section{Two DDR-based schemes for the Yang--Mills equations}\label{sec:schemes}

We propose two schemes for the Yang--Mills equations, which only differ in the handling of the nonlinear terms appearing in the equations. The first was introduced at the lowest order in \cite{Droniou.Oliynyk.ea:23}, in which a discrete `bracket' was constructed to approximate the value in the $\LaXdiv{k}{h}$ space. This term is used in the discrete $L^2$-products, and the exact preservation of a discrete constraint, as well as energy estimates, are proven. Numerical tests in \cite{Droniou.Oliynyk.ea:23} however only considered the lowest-order $k=0$ of the method.

The second method we present here is new, and leverages instead the continuous $L^2$-product and nonlinear bracket, as well as the elemental potential reconstructions, to achieve the same goal.

\subsection{Weak constrained form of the equations}

Deriving from the bracket on the Lie algebra, the following two bilinear maps are defined:
\begin{align}
	[\cdot,\cdot]:{}&(\mathfrak X(U)\otimes\La)\times(C^\infty(U)\otimes\La)\to\mathfrak X(U)\otimes\La,&[\bvec{v},q]\mapsto{}&\bvec{v}^Iq^J\otimes[e_I,e_J],\label{cont.bkt.1}\\
	\ebkt{\cdot}{\cdot}:{}&(\mathfrak X(U)\otimes\La)\times(\mathfrak X(U)\otimes\La)\to\mathfrak X(U)\otimes\La,&\ebkt{\bvec{v}}{\bvec{w}}\mapsto{}&(\bvec{v}^I\times\bvec{w}^J)\otimes[e_I,e_J].\label{cont.bkt.2}
\end{align}
We use in the discretisation a weak \emph{constrained} formulation of the Yang--Mills equations, appearing previously in \cite{Christiansen.Winther:06}: Find $(\bvec{A},\bvec{E},\lambda):[0,T]\to(\LaHcurl{U})^2\times (\LaHgrad{U})$ such that
\begin{subequations}\label{eq:lm.ym}
\begin{alignat}{4}
\partial_t\bvec{A}={}&-\bvec{E},\label{eq:lm.ym.0}\\
\int_U\langle\partial_t\bvec{E},\bvec{v}\rangle+\int_U\langle\GRAD\lambda+{}&[\bvec{A},\lambda],\bvec{v}\rangle=\int_U\langle\CURL\bvec{A},\CURL \bvec{v}\rangle+\int_U\langle\CURL\bvec{A},\ebkt{\bvec{A}}{\bvec{v}}\rangle\nonumber\\
&+\int_U\left\langle\frac12\ebkt{\bvec{A}}{\bvec{A}},\CURL \bvec{v}+\ebkt{\bvec{A}}{\bvec{v}}\right\rangle,\qquad\forall\bvec{v}\in\LaHcurl{U},
\label{eq:lm.ym.1}\\
\label{eq:lm.ym.2}
\int_U\langle\partial_t\bvec{E},\GRAD q+[\bvec{A},q]\rangle={}&0,\qquad\forall q\in \LaHgrad{U}.
\end{alignat}
\end{subequations} 
Note that the right-hand side of \eqref{eq:lm.ym.1} is equal to $\int_U\langle\bvec{B},\CURL \bvec{v}+\ebkt{\bvec{A}}{\bvec{v}}\rangle$, where the magnetic field is defined as $\bvec{B}\coloneq\CURL\bvec{A}+\frac12\ebkt{\bvec{A}}{\bvec{A}}$. We have developed this expression as it will drive different choices of discretisations. Solutions to these equations preserve the quantity
\begin{equation}
\int_U\langle\bvec{E},\GRAD q+[\bvec{A},q]\rangle, \qquad\forall q\in H^1(U)\otimes\La.
\end{equation}
The use of this particular constrained form facilitates the preservation of a discrete counterpart in the numerical scheme; discussions on the derivation and implications of these continuous equations can be found in \cite{Christiansen.Winther:06,Droniou.Oliynyk.ea:23}.

\subsection{Schemes}

We consider a time discretisation $0=t^0<t^1<\ldots<t^N=T$ of $[0,T]$ and denote the step in time between $n$ and $n+1$ as $\deltat^{n+\frac{1}{2}}\coloneq t^{n+1}-t^n$. Then define for a family $v=(v^n)_n$, 
\[
  \delta_t^{n+1}v=\frac{v^{n+1}-v^n}{\deltat^{n+\frac{1}{2}}}.
\]
Starting from initial conditions $(\uvec{A}_h^0,\uvec{E}_h^0)\in (\LaXcurl{k}{h})^2$, the constrained scheme based on \eqref{eq:lm.ym} is: Find families $(\uvec{A}_h^n)_n$, $(\underline{\bvec{E}}_h^n)_n$, $(\underline{\lambda}_h^n)_n$ such that for all $n$, $(\uvec{A}_h^n,\underline{\bvec{E}}_h^n,\underline{\lambda}_h^n)\in(\LaXcurl{k}{h})^2\times(\LaXgrad{k}{h})$ and
\begin{subequations} \label{eq:ym.lm.scheme}
\begin{alignat}{4}
\delta_t^{n+1}\uvec{A}_h&=-\uvec{E}_h^{n+1},\label{eq:ym.lm.scheme.1}\\ 
(\delta_t^{n+1}\uvec{E}_h,\uvec{v}_h)_{\CURL,\La,h}+(\LauGh{k}{\underline{\lambda}_h^{n+1}},{}&\uvec{v}_h)_{\CURL,\La,h}+\int_U\langle[\LaPcurlh{k}\uvec{A}_h^{n+1},\LaPgradh{k+1}\underline{\lambda}_h^{n+1}],\LaPcurlh{k}\uvec{v}_h\rangle\nonumber\\
=(\LauCh{k}\uvec{A}_h^{n+1},{}&\LauCh{k}\uvec{v}_h)_{\DIV,\La,h}+\mathfrak{N}(\uvec{A}_h^{n},\uvec{A}_h^{n+1};\uvec{v}_h),\qquad\forall\uvec{v}_h\in\LaXcurl{k}{h},\label{eq:ym.lm.scheme.2}\\
(\delta_t^{n+1}\uvec{E}_h,\LauGh{k}\underline{q}_h)_{\CURL,\La,h}{}&+\int_U\langle\LaPcurlh{k}(\delta_t^{n+1}\uvec{E}_h),[\LaPcurlh{k}\uvec{A}_h^{n},\LaPgradh{k+1}\underline{q}_h]\rangle=0\nonumber\\
&\qquad\qquad\qquad\qquad\forall\underline{q}_h\in\LaXgrad{k}{h}.\label{eq:ym.lm.scheme.3} 
\end{alignat} 
\end{subequations}
We refer the reader to \cite[Section 4]{Droniou.Oliynyk.ea:23} for a discussion on the choice of the initial conditions $(\uvec{A}_h^0,\uvec{E}_h^0)$, and also for alternative choices to the fully implicit time stepping selected here.

In \eqref{eq:ym.lm.scheme.2}, $\mathfrak{N}\in\{\mathfrak{N}_1,\mathfrak{N}_2\}$ is one of the following two discretisations of the nonlinear terms in the right-hand side of \eqref{eq:lm.ym.1}:
\begin{align}
\mathfrak{N}_1(\uvec{A}_h^{n},\uvec{A}_h^{n+1};\uvec{v}_h){}&\coloneq(\LauCh{k}\uvec{A}_h^{n+1},\ebkttrk{\uvec{A}_h^{n+\frac{1}{2}}}{\uvec{v}_h})_{\DIV,\La,h}\nonumber\\
&+\left(\frac12\ebkttrk{\uvec{A}_h^{n+1}}{\uvec{A}_h^{n+1}},\LauCh{k}\uvec{v}_h+\ebkttrk{\uvec{A}_h^{n+\frac{1}{2}}}{\uvec{v}_h}\right)_{\DIV,\La,h},\label{disc.N.1}\\
\mathfrak{N}_2(\uvec{A}_h^{n},\uvec{A}_h^{n+1};\uvec{v}_h){}&\coloneq\int_U\langle\cCh{k,\La}\uvec{A}_h^{n+1},\ebkt{\LaPcurlh{k}\uvec{A}_h^{n+\frac12}}{\LaPcurlh{k}\uvec{v}_h}\rangle\nonumber\\
&+\int_U\Big\langle\frac12\ebkt{\LaPcurlh{k}\uvec{A}_h^{n+1}}{\LaPcurlh{k}\uvec{A}_h^{n+1}},\cCh{k,\La}\uvec{v}_h+\ebkt{\LaPcurlh{k}\uvec{A}_h^{n+\frac12}}{\LaPcurlh{k}\uvec{v}_h}\Big\rangle,\label{disc.N.2}
\end{align}
Above, we have set $\uvec{A}_h^{n+\frac{1}{2}}=\frac12(\uvec{A}_h^n+\uvec{A}_h^{n+1})$. Moreover, in $\mathfrak{N}_1$, we have made use of the discrete version $\ebkttrk{\cdot}{\cdot}:(\LaXcurl{k}{h})^2\to\LaXdiv{k}{h}$ of the map \eqref{cont.bkt.2}, defined for all $\uvec{v}_h,\uvec{w}_h\in\LaXcurl{k}{h}$ through its components by:
\begin{subequations}\label{eq:def.ebkttrk} 
\begin{alignat}{2}
\left(\ebkttrk{\uvec{v}_h}{\uvec{w}_h}\right)_F={}&\lproj{k}{F}\left(\ebkt{\LatrFt{k}\uvec{v}_F}{\LatrFt{k}\uvec{w}_F}\cdot\normal_F\right)\in\Poly{k}(F)\otimes\La\quad
\forall F\in\Fh, 
\label{eq:def.ebkttrk.defF}\\
\left(\ebkttrk{\uvec{v}_h}{\uvec{w}_h}\right)_{\cvec{G},T}={}&\Gproj{k-1}{T}\left(\ebkt{\LaPcurl{k}\uvec{v}_T}{\LaPcurl{k}\uvec{w}_T}\right)\in\Goly{k-1}(T)\otimes \La\quad
\forall T\in\Th,
\label{eq:def.ebkttrk.def.GolyT}\\
\left(\ebkttrk{\uvec{v}_h}{\uvec{w}_h}\right)_{\cvec{G},T}^\compl={}&\Gcproj{k}{T}\left(\ebkt{\LaPcurl{k}\uvec{v}_T}{\LaPcurl{k}\uvec{w}_T}\right)\in\cGoly{k}(T)\otimes\La\quad
\forall T\in\Th.
\label{eq:def.ebkttrk.def.cGolyT}
\end{alignat}
\end{subequations}
In $\mathfrak N_2$ we have used the global piecewise polynomial curl $\cCh{k,\La}$ defined by patching the element curls: $(\cCh{k,\La}\uvec{v}_h)_{|T}=\cCT{k,\La}\uvec{v}_T$ for all $T\in\Th$ and $\uvec{v}_h\in\LaXcurl{k}{h}$.

\begin{remark}[Motivation for the discretisation of the nonlinear terms]
The nonlinear terms in the right-hand side of \eqref{eq:lm.ym.1} are
\begin{equation}\label{eq:nl.terms}
\int_U\langle\CURL\bvec{A},\ebkt{\bvec{A}}{\bvec{v}}\rangle
+\int_U\left\langle\frac12\ebkt{\bvec{A}}{\bvec{A}},\CURL \bvec{v}+\ebkt{\bvec{A}}{\bvec{v}}\right\rangle.
\end{equation}
When discretising these terms, the continuous fields $\bvec{A},\bvec{v}\in\LaHcurl{U}$ are replaced by fully discrete objects $\uvec{A}_h,\uvec{v}_h\in\LaXcurl{k}{h}$, and we have to give meaning to the terms in \eqref{eq:nl.terms} after this substitution -- which is not straightforward since $\LaXcurl{k}{h}$ is not a subspace of $\LaHcurl{U}$.

Applying the standard DDR procedure on \eqref{eq:nl.terms}, we build these terms by replacing the inner product $\int_U\langle\cdot,\cdot\rangle$ and differential $\CURL$ by the corresponding discrete notions found in the LASDDR complex. The only missing element is a discrete version of the bracket $\ebkt{\cdot}{\cdot}$ on $\LaXcurl{k}{h}\times \LaXcurl{k}{h}$ which produces consistent discrete approximations in $\LaXdiv{k}{h}$. This is what \eqref{eq:def.ebkttrk} provides, and this approach leads to $\mathfrak N_1$.

Another approach to discretising \eqref{eq:nl.terms} is in a sense more straightforward (but only works because none of these terms, in the weak formulation, comes from integrating-by-parts the strong form of the model): since we can reconstruct piecewise polynomial reconstructions and curls from elements in $\LaXcurl{k}{h}$, we can decide to simply substitute all the terms $\bvec{A},\bvec{v}$ by these polynomial reconstruction based on $\uvec{A}_h,\uvec{v}_h$ and keep the other elements (integrals, brackets) exactly the same. This idea leads to $\mathfrak N_2$.
\end{remark}

\begin{remark}[Discretisation of the linear terms]
The same way we used, in $\mathfrak N_2$, the piecewise polynomial potentials and element curl, we could consider replacing, in \eqref{eq:ym.lm.scheme.2}, the term $(\LauCh{k}\uvec{A}_h^{n+1},\LauCh{k}\uvec{v}_h)_{\DIV,\La,h}$ with $\int_U \langle\cCh{k,\La}\uvec{A}_h^{n+1},\cCh{k,\La}\uvec{v}_h\rangle$. This would however not lead to a suitable scheme, for the following reason.

Consider the pure Maxwell model, discretised using a linear unconstrained scheme (that is, \eqref{eq:ym.lm.scheme.1}--\eqref{eq:ym.lm.scheme.2} without the terms involving $\underline{\lambda}_h$, without the nonlinear terms and with $\La=\Real$ with the trivial Lie bracket):
\begin{subequations} \label{eq:mx.scheme}
\begin{alignat}{4}
\delta_t^{n+1}\uvec{A}_h&=-\uvec{E}_h^{n+1},\label{eq:mx.scheme.1}\\ 
(\delta_t^{n+1}\uvec{E}_h,\uvec{v}_h)_{\CURL,\La,h}&=(\LauCh{k}\uvec{A}_h^{n+1},{}&\LauCh{k}\uvec{v}_h)_{\DIV,\La,h}.\label{eq:mx.scheme.2}
\end{alignat} 
\end{subequations}
Using the results in \cite[Section 6]{Di-Pietro.Droniou:23*1} it can easily be shown that, for a smooth enough potential $\bvec{A}$, solution of the continuous model, the consistency error (as defined in \cite{Di-Pietro.Droniou:18}) of the scheme satisfies 
\begin{equation}\label{eq:Eh.maxwell}
\mathcal E_h(\bvec{A};\uvec{v}_h)\le C_{\bvec{A}}(\deltat+h^{k+1})\left(\norm{\CURL,\La,h}{\uvec{v}_h}+\norm{\DIV,\La,h}{\uCh{k}\uvec{v}_h}\right),
\end{equation}
where $\norm{\CURL,\La,h}{{\cdot}}$ and $\norm{\DIV,\La,h}{{\cdot}}$ denote the norms respectively associated with the inner products $(\cdot,\cdot)_{\CURL,\La,h}$ and $(\cdot,\cdot)_{\DIV,\La,h}$.
The scheme \eqref{eq:mx.scheme} is stable for the norm $\norm{\CURL,\La,h}{{\cdot}}+\norm{\DIV,\La,h}{\uCh{k}{\cdot}}$, so \eqref{eq:Eh.maxwell} and the 3rd Strang Lemma \cite{Di-Pietro.Droniou:18} provide an $\mathcal O(\deltat +h^{k+1})$ error estimate.

However, replacing $(\LauCh{k}\uvec{A}_h^{n+1},\LauCh{k}\uvec{v}_h)_{\DIV,\La,h}$ with $\int_U \langle\cCh{k,\La}\uvec{A}_h^{n+1},\cCh{k,\La}\uvec{v}_h\rangle$ in \eqref{eq:mx.scheme.2} results in a scheme that is stable for the weaker norm $\norm{\CURL,\La,h}{{\cdot}}+\norm{L^2(U)}{\cCh{k,\La}\cdot}$ (which does not control, in particular, the face curls). On the other hand, the consistency estimate remains \eqref{eq:Eh.maxwell}, in the stronger norm. As a consequence, this estimate and the weaker stability cannot be combined together to obtain error estimates on the scheme. As a matter of fact, numerical tests (not reported in this paper) show that, on some mesh families, this alternative scheme does not converge as the mesh size and time step are refined.
\end{remark}

\subsection{Discrete energy and constraint preservation}

We define the discrete conserved quantity through the constraint functional $\mathfrak C^n:\LaXgrad{k}{h}\to\Real$:
\begin{equation}\label{discrete.const}
\begin{aligned}
\mathfrak C^n(\underline{q}_h)\coloneq(\uvec{E}_h^n,\LauGh{k}\underline{q}_h)_{\CURL,\La,h}+\int_U\langle\LaPcurlh{k}\uvec{E}_h^n,{}&[\LaPcurlh{k}\uvec{A}_h^n,\LaPgradh{k+1}\underline{q}_h]\rangle\\
&\forall\underline{q}_h\in\LaXgrad{k}{h}.
\end{aligned}
\end{equation}

\begin{proposition}[Constraint preservation]\label{prop:constraint.preservation}
For any choice of $\mathfrak{N}$, if $(\uvec{A}_h^n,\uvec{E}_h^n,\underline{\lambda}_h^n)$ solve \eqref{eq:ym.lm.scheme} then, for all $\underline{q}_h\in\LaXgrad{k}{h}$, the quantity $\mathfrak C^n(\underline{q}_h)$ is independent of $n$.
\end{proposition}

\begin{proof}
We note that the proof for the preservation of constraint in \cite[Proposition 7]{Droniou.Oliynyk.ea:23} is in fact independent of the discretisation of the second equation \eqref{eq:ym.lm.scheme.2}, and also applies in the case for general $k$ (see \cite[Section 6.2]{Droniou.Oliynyk.ea:23}).
\end{proof}

To state the energy dissipation property, we introduce the discrete magnetic fields based on $\bvec{B}$. Their nature depends on the chosen discretisation of the nonlinear terms in \eqref{eq:lm.ym.1}. If $\mathfrak N=\mathfrak N_1$, exploiting the discrete bracket $\ebkttrk{\cdot}{\cdot}$ we can define the discrete magnetic field as an element of $\LaXcurl{k}{h}$:
\[
\uvec{B}_{h,1}^n\coloneq\LauCh{k}\uvec{A}_h^n+\frac12\ebkttrk{\uvec{A}_h^n}{\uvec{A}_h^n}\in \LaXcurl{k}{h}.
\]
If $\mathfrak N=\mathfrak N_2$, the nonlinear terms being discretised as piecewise polynomial functions, the discrete magnetic field has the same nature:
\begin{equation}\label{eq:def.B2}
\bvec{B}_{h,2}^n\coloneq\cCh{k,\La}\uvec{A}_h^n+\frac12\ebkt{\LaPcurlh{k}\uvec{A}_h^n}{\LaPcurlh{k}\uvec{A}_h^n}\in\Poly{k}(\Th)\otimes\La.
\end{equation}

\begin{proposition}[Energy dissipation]\label{prop:energy.ym}
If the initial conditions $(\uvec{A}_h^0,\uvec{E}_h^0)$ are such that $\mathfrak C^0\equiv 0$, then we have the decay of energy in the sense that, for all $n$, 
\begin{equation}\label{eq:lm.ym.energy}
\frac{1}{2}\norm{\CURL,\La,h}{\uvec{E}_h^{n+1}}^2+\mathfrak B^{n+1}\leq\frac{1}{2}\norm{\CURL,\La,h}{\uvec{E}_h^n}^2+\mathfrak B^n,
\end{equation}
where 
\[
\mathfrak B^n\coloneq\left\{
\begin{array}{ll}
\displaystyle\frac{1}{2}\norm{\DIV,\La,h}{\uvec{B}_{h,1}^n}^2&\quad\mbox{ if }\quad\mathfrak N=\mathfrak N_1,\\[1em]
\displaystyle\frac{1}{2}\norm{L^2(U)\otimes\La}{\bvec{B}_{h,2}^n}^2+\frac12\mathrm{s}_{\DIV,h}^{\La}(\LauCh{k}\uvec{A}_h^n,\LauCh{k}\uvec{A}_h^n)
&\quad\mbox{ if }\quad\mathfrak N=\mathfrak N_2,
\end{array}
\right.
\]
where $\mathrm{s}_{\DIV,h}^{\La}$ is the stabilisation form involved in the definition of the inner product $(\cdot,\cdot)_{\DIV,\La,h}$ (that is, the tensorisation of $\mathrm{s}_{\DIV,h}$ defined by \eqref{eq:def.sdiv}).
\end{proposition}
\begin{proof}
The proof for $\mathfrak N=\mathfrak N_1$ is identical to the one for $k=0$ done in \cite[Proposition 8]{Droniou.Oliynyk.ea:23} (see also \cite[Section 6.2]{Droniou.Oliynyk.ea:23}). 

Let us consider the case $\mathfrak N=\mathfrak N_2$.
Choosing $\uvec{v}_h=\uvec{E}_h^{n+1}$ in \eqref{eq:ym.lm.scheme.2} and multiplying by $\deltat^{n+\frac12}$, the first term in the LHS is
\[
(\uvec{E}_h^{n+1}-\uvec{E}_h^n,\uvec{E}_h^{n+1})_{\CURL,\La,h}=\frac{1}{2}\norm{\CURL,\La,h}{\uvec{E}_h^{n+1}}^2-\frac{1}{2}\norm{\CURL,\La,h}{\uvec{E}_h^n}^2+\frac{1}{2}\norm{\CURL,\La,h}{\uvec{E}_h^{n+1}-\uvec{E}_h^n}^2,
\]
while the remaining two form the constraint $\mathfrak C^{n+1}(\underline{\lambda}_h^{n+1})$, that vanishes by Proposition \ref{prop:constraint.preservation} and the assumption that $\mathfrak C^0\equiv 0$. 

On the RHS, we use \eqref{eq:ym.lm.scheme.1} to substitute instead $\uvec{v}_h=-\delta_t^{n+1}\uvec{A}_h$, noting the cancellation of $\deltat^{n+\frac12}$ after multiplying. Expanding the discrete $L^2$-product by its definition and invoking \cite[Proposition 7]{Di-Pietro.Droniou:23*1} (which can easily be extended to the SDDR complex using \cite[Eq.~(2.2)]{Di-Pietro.Droniou:23*2}) to write $\LaPdivh{k}\LauCh{k}=\cCh{k,\La}$, the first term is
\begin{align*}
(\LauCh{k}\uvec{A}_h^{n+1},\LauCh{k}(\uvec{A}_h^{n}-\uvec{A}_h^{n+1}))_{\DIV,\La,h}={}&\int_U\langle\cCh{k,\La}\uvec{A}_h^{n+1},\cCh{k,\La}(\uvec{A}_h^n-\uvec{A}_h^{n+1})\rangle\\
&+\mathrm{s}_{\DIV,h}^{\La}(\LauCh{k}\uvec{A}_h^n,\LauCh{k}(\uvec{A}_h^{n}-\uvec{A}_h^{n+1})),
\end{align*}
Combining with the integrals in $\mathfrak N_2$ (see \eqref{disc.N.2}), expanding $\uvec{A}_h^{n+\frac12}=\frac12(\uvec{A}_h^n+\uvec{A}_h^{n+1})$, recalling the definition \eqref{eq:def.B2} of $\bvec{B}_{h,2}$, then using the symmetry and bilinearity of the bracket \eqref{cont.bkt.2}, the RHS becomes
\begin{align*}
(\LauCh{k}{}&\uvec{A}_h^{n+1},\LauCh{k}(\uvec{A}_h^{n}-\uvec{A}_h^{n+1}))_{\DIV,\La,h}+\mathfrak{N}_2(\uvec{A}_h^{n},\uvec{A}_h^{n+1};\uvec{A}_h^{n}-\uvec{A}_h^{n+1})\\
={}&\int_U\left\langle\bvec{B}_h^{n+1},\cCh{k,\La}(\uvec{A}_h^{n}-\uvec{A}_h^{n+1})+\ebkt{\LaPcurlh{k}\frac12(\uvec{A}_h^{n}+\uvec{A}_h^{n+1})}{\LaPcurlh{k}(\uvec{A}_h^{n}-\uvec{A}_h^{n+1})}\right\rangle\\
&+\mathrm{s}_{\DIV,h}^{\La}(\LauCh{k}\uvec{A}_h^{n+1},\LauCh{k}(\uvec{A}_h^{n}-\uvec{A}_h^{n+1}))\\
={}&\int_U\left\langle\bvec{B}_h^{n+1},\cCh{k,\La}\uvec{A}_h^{n}+\frac12\ebkt{\LaPcurlh{k}\uvec{A}_h^{n}}{\LaPcurlh{k}\uvec{A}_h^{n}}-\cCh{k,\La}\uvec{A}_h^{n+1}-\frac12\ebkt{\LaPcurlh{k}\uvec{A}_h^{n+1}}{\LaPcurlh{k}\uvec{A}_h^{n+1}}\right\rangle\\
&+\mathrm{s}_{\DIV,h}^{\La}(\LauCh{k}\uvec{A}_h^{n+1},\LauCh{k}(\uvec{A}_h^{n}-\uvec{A}_h^{n+1}))\\
={}&\int_U\left\langle\bvec{B}_h^{n+1},\bvec{B}_h^n-\bvec{B}_h^{n+1}\right\rangle+\mathrm{s}_{\DIV,h}^{\La}(\LauCh{k}\uvec{A}_h^{n+1},\LauCh{k}(\uvec{A}_h^{n}-\uvec{A}_h^{n+1}))\\
={}&\frac12\Big(\norm{L^2(U)\otimes\La}{\bvec{B}_h^n}^2-\norm{L^2(U)\otimes\La}{\bvec{B}_h^{n+1}}^2-\norm{L^2(U)\otimes\La}{\bvec{B}_h^{n+1}-\bvec{B}_h^n}^2+\mathrm{s}_{\DIV,h}^{\La}(\LauCh{k}\uvec{A}_h^n,\LauCh{k}\uvec{A}_h^n)\\&-\mathrm{s}_{\DIV,h}^{\La}(\LauCh{k}\uvec{A}_h^{n+1},\LauCh{k}\uvec{A}_h^{n+1})-\mathrm{s}_{\DIV,h}^{\La}(\LauCh{k}\uvec{A}_h^{n+1}-\LauCh{k}\uvec{A}_h^n,\LauCh{k}\uvec{A}_h^{n+1}-\LauCh{k}\uvec{A}_h^n)\Big),
\end{align*}
where the conclusion follows by applying the relation $b(x,y-x)=\frac12 b(y,y)-\frac12 b(x,x)-\frac12 b(y-x,y-x)$ to the symmetric bilinear forms $b=(\cdot,\cdot)_{L^2(U)\otimes \La}$ and $b=\mathrm{s}_{\DIV,h}^\La$. 

Finally arranging both side of the equation, moving the pure $n+1$ terms to the left and the rest to the right, we get
\begin{align*}
\frac{1}{2}\norm{\CURL,\La,h}{\uvec{E}_h^{n+1}}^2+\mathfrak B^{n+1}={}&\frac{1}{2}\norm{\CURL,\La,h}{\uvec{E}_h^n}^2+\mathfrak B^n\\
&-\frac{1}{2}\norm{\CURL,\La,h}{\uvec{E}_h^{n+1}-\uvec{E}_h^n}^2-\frac12\norm{L^2(U)\otimes\La}{\bvec{B}_h^{n+1}-\bvec{B}_h^n}^2\\
&-\frac12\mathrm{s}_{\DIV,h}^{\La}(\LauCh{k}\uvec{A}_h^{n+1}-\LauCh{k}\uvec{A}_h^n,\LauCh{k}\uvec{A}_h^{n+1}-\LauCh{k}\uvec{A}_h^n),
\end{align*}
proving the statement, since the norm and $\mathrm{s}_{\DIV,h}^{\La}$ are both positive semidefinite. 
\end{proof}

%% Implementation
\section{Implementation}\label{sec:implementation}

We cover in this section the broad mechanisms of how the schemes are implemented. For our numerical simulations, this implementation was done in the \texttt{HArDCore3D} library (see \url{https://github.com/jdroniou/HArDCore}) starting from the serendipity DDR spaces and operators described in Section \ref{sec:ddr}. This library contains a fully automated construction of these objects, including the computation of the degree depletions $\ell_\sfP$ defined at the start of Section \ref{sec:SDDR.complex}.

We first eliminate $\uvec{A}_h^{n+1}$ by using the first equation \eqref{eq:ym.lm.scheme.1} to write $\uvec{A}_h^{n+1}=\uvec{A}_h^n-\deltat^{n+\frac12}\uvec{E}_h^{n+1}$. Denoting the resulting equation by $F(\bvec{X}^{n+1})=\bvec{b}$, where $\bvec{X}^{n+1}$ represents the combined vector of $\uvec{E}_h^{n+1},\lambda_h^{n+1}$, we then employ the iterative Newton method to find a solution up to an accuracy of $\epsilon$. The quantity $\uvec{A}_h^{n+1}$, which is required at the next time step, is finally recovered via back substitution. 

At a single time step $n$, the most costly part of this process lies in the repeated assembly and resolution of the linear Newton problem: Find vectors $(\bvec{X}^{n+1,i+1})_{i\in\mathbb{N}}$ such that
\[
  DF_{\bvec{X}^{n+1,i}}(\bvec{X}^{n+1,i+1}-\bvec{X}^{n+1,i})=\bvec{b}-F(\bvec{X}^{n+1,i})\quad\text{ until }\quad\frac{\norm{l^2}{F(\bvec{X}^{n+1,i+1})-\bvec{b}}}{\norm{l^2}{\bvec{b}}}\leq\epsilon.
\]
Although the derivative matrix $DF_{\bvec{X}^{n+1,i}}$ is simple to determine because $F$ is multilinear, it must be rebuilt at every iteration, with terms stemming from bilinear, trilinear, and even quadrilinear forms (see the product of bilinear brackets in the last terms of \eqref{disc.N.1} and \eqref{disc.N.2}). In the rest of this section, we discuss how to perform these calculations without it becoming too expensive in either memory space or computational time.

The other major expense is resolving the linear system, for which we use the \texttt{Intel MKL PARADISO} library (see \url{https://software.intel.com/en-us/mkl}), which provides a multi-threaded direct solver. An efficient technique to reduce this solver cost is to apply to the linear systems the static condensation process, which eliminates all elemental unknowns of the system prior to solving. This is made possible by the specific stencil resulting from a hybrid method like (S)DDR, which couples the unknowns inside one element only with the unknowns on the faces, edges and vertices of that element (inter-element unknowns are never directly coupled). We emphasize here the difference between static condensation and the serendipity DDR process. Both reduce the number of degrees of freedom, but the SDDR spaces and operators are leaner from the start; as a result, any scheme built from it will see an improved performance in every aspect of the implementation at no additional cost. In contrast, static condensation reduces the number of unknowns only after all contributions are assembled, and thus its effect is more limited as it only reduces the cost of solving the global system, not the cost of assembling that system. Additionally, it must be repeated every iteration, with some overhead (solving a smaller linear system in each element) each time. Finally, we should highlight that, since static condensation only eliminates unknowns in the elements while serendipity also eliminates degrees of freedom on the edges/faces, the linear systems resulting from a statically condensed DDR scheme remain larger in general than the linear systems resulting from a statically condensed SDDR scheme.

\subsection{LASDDR tensorisation}

The numerical construction of the LASDDR complex primarily consists of wrapping a layer of matrix tensorisation around the existing SDDR code. For a code like \texttt{HArDCore3D} based on the \texttt{Eigen3} library (see \url{http://eigen.tuxfamily.org}), this is easily achieved using the \texttt{KroneckerProduct} functionality.

Fixing a basis $(e_I)_I$ of the $d$-dimensional Lie algebra $\La$, each Lie algebra-valued degree of freedom can by expressed by $d$ real values. In other words, we can think of an element of an LASDDR space as being made up of $d$ SDDR vectors, one for each basis $e_I$; i.e. $\uvec{v}_h\equiv\uvec{v}_h^I\otimes e_I$. Fixing an ordering that combines everything into a single vector fixes the physical interpretation of all the remaining operators. We choose to store the values associated to each mesh entity (vertex, edge, face, element) sequentially, but for ease of distinguishing the significance of each entry, they are doubly indexed: $(\uvec{v}_h^i)^I$. The lowercase letter numbers the mesh entity it originates from, and the capital letter labels the Lie algebra basis it is attached to. As an example, if $d=3$, we have the following structures:
\begin{align*}
\uvec{v}_h=
\begin{bmatrix}
(\uvec{v}_h^1)^1 \\[.3em]
(\uvec{v}_h^1)^2 \\[.3em]
(\uvec{v}_h^1)^3 \\[.3em]
(\uvec{v}_h^2)^1 \\
\vdots                   
\end{bmatrix}, \quad
\uvec{v}_h^i=
\begin{bmatrix}
(\uvec{v}_h^i)^1 \\[.3em]
(\uvec{v}_h^i)^2 \\[.3em]
(\uvec{v}_h^i)^3 \\[.3em]                 
\end{bmatrix}, \quad
\uvec{v}_h^I=
\begin{bmatrix}
(\uvec{v}_h^1)^I \\[.3em]
(\uvec{v}_h^2)^I \\[.3em]
(\uvec{v}_h^3)^I \\[.3em]
(\uvec{v}_h^4)^I \\
\vdots                   
\end{bmatrix}.
\end{align*}

Then, considering a linear LASDDR operator $L^\La$, the matrix representation $\boldsymbol{\sf L}^{\La}$ must by definition act as $\boldsymbol{\sf L}^{\La}(\uvec{v}_h)\equiv\boldsymbol{\sf L}(\uvec{v}_h^I)\otimes e_I$, where $\boldsymbol{\sf L}$ is the corresponding SDDR matrix operator. From some basic arithmetic, we can conclude that $\boldsymbol{\sf L}^{\La}$ has the form
\begin{align*}
\boldsymbol{\sf L}^{\La}=\boldsymbol{\sf L}\otimes\Id{d}=
\begin{bmatrix}
L_{11}\Id{d} &L_{12}\Id{d} &\cdots\\
L_{21}\Id{d} &L_{22}\Id{d} &\\
\vdots                   &                         &\ddots
\end{bmatrix},
\end{align*}
where $\Id{d}$ is the $d\times d$ identity matrix. For bilinear operators such as the discrete inner products $(\cdot,\cdot)_{\bullet,\La,h}$, and the bracket terms which we deal with in the next section, the idea is very much the same; the difference lies in the usage of a more general $d\times d$ matrix $\boldsymbol{\sf M}$ in place of $\Id{d}$, indicating an interaction of the Lie algebra bases. For example, denoting the LASDDR (resp. SDDR) product matrix by $\boldsymbol{\sf B}^\La$ (resp. $\boldsymbol{\sf B}$), with action previously defined as $(\uvec{v}_h)^T\boldsymbol{\sf B}^\La(\uvec{w}_h)=((\uvec{v}_h^I)^T\boldsymbol{\sf B}(\uvec{w}_h^J))\langle e_I,e_J\rangle$, the $\boldsymbol{\sf M}$ becomes evidently the mass matrix of the Lie algebra:
\begin{align}\label{matrix.tensor}
\boldsymbol{\sf M}=
\begin{bmatrix}
\langle e_1,e_1\rangle &\langle e_1,e_2\rangle &\cdots\\
\langle e_2,e_1\rangle &\langle e_2,e_2\rangle &\\
\vdots                   &                     &\ddots
\end{bmatrix}, \quad
\boldsymbol{\sf B}^{\La}=\boldsymbol{\sf B}\otimes\boldsymbol{\sf M}=
\begin{bmatrix}
B_{11}\boldsymbol{\sf M} &B_{12}\boldsymbol{\sf M} &\cdots\\
B_{21}\boldsymbol{\sf M} &B_{22}\boldsymbol{\sf M} &\\
\vdots                   &                         &\ddots
\end{bmatrix}.
\end{align}

\subsection{Bracket terms}

The nonlinear terms in both schemes create operators which can not be encoded by a single matrix, but rather (local or global) 3 or 4-dimensional arrays: 
\begin{align*}
\int_T\langle\LaPcurl{k}\cdot,{}&[\LaPcurl{k}\cdot,\LaPgrad{k+1}\cdot]\rangle,& \\
(\cdot,{}&\ebkttrk{\cdot}{\cdot})_{\DIV,\La,h},& (\ebkttrk{\cdot}{\cdot},{}&\ebkttrk{\cdot}{\cdot})_{\DIV,\La,h},\\
\int_T\langle\cCT{k,\La}\cdot,{}&\ebkt{\LaPcurl{k}\cdot}{\LaPcurl{k}\cdot}\rangle, & \int_T\langle\ebkt{\LaPcurl{k}\cdot}{\LaPcurl{k}\cdot},{}&\ebkt{\LaPcurl{k}\cdot}{\LaPcurl{k}\cdot}\rangle.
\end{align*}
The tools available in \texttt{Eigen3} for dealing with these objects are not nearly as developed as the ones for matrices. In most cases, the multidimensional storage is done using the \texttt{Boost.MultiArray} library (see \url{https://www.boost.org/doc/libs/1_61_0/libs/multi_array/doc/index.html}), but the \texttt{Eigen::Map} function is used to interpret the data, so that we can still perform the usual matrix operations.

These generic objects are independent of time, so seemingly the most time efficient method would be to pre-compute them once and for all, and recall them when necessary. Unfortunately, the extra dimensionalities mean these operators are much larger than the (bi)linear ones appearing in LASDDR; in fact the memory usage to store these terms grows exponentially with the number of entries. The Lie algebra tensorisation only exacerbates this problem, even accounting for the many symmetries that could be exploited. This memory issue is particularly sensitive in an implementation -- such as ours (which follows the HArDCore general strategy) -- that assumes that each element can have its own geometry, which forces the local arrays to be computed/stored independently for each element (in a situation where the mesh elements can be classified using a few reference elements, all memory issues disappear as only local multilinear maps in reference elements need to be stored). In this context, the immense amount of memory required to store just a single one of these global maps means that there is no choice but to recalculate them at each time step and locally (mesh entity by mesh entity) as needed. 

Another important effect on the runtime lies in the order in which some tensorisation-related computations are performed. Taking the sum of smaller matrices multiple times is sometimes preferable to doing it once with larger matrices (which often have lots of zeros). Therefore, delaying the tensorisation until after the vectors have been evaluated can improve both the memory usage and the speed, even though some calculations and the tensorisation have to be repeated.

We furnish these ideas with an example for the nonlinear terms of the form 
\[
\int_T\langle\LaPcurl{k}\cdot,[\LaPcurl{k}\uvec{v}_T,\LaPgrad{k+1}\cdot]\rangle,
\]
in which $\uvec{v}_T$ is a given vector in $\LaXcurl{k}{T}$.
This term is represented by a coefficient matrix with entries (doubly indexed by $((i,I),(k,K))$) given by:
\[
\int_T\langle\LaPcurl{k}(\phi_i)_I,[\LaPcurl{k}\uvec{v}_T,\LaPgrad{k+1}(\psi_k)_K]\rangle,
\]
where the rows and columns range over the basis vectors, that are defined as $(\phi_i)_I\equiv\phi_i\otimes e_I$ (resp. $(\psi_k)_K)$ where $(\phi_i)_i$ is a basis for $\Xcurl{k}{T}$ (resp. $(\psi_k)_k$ a basis of $\Xgrad{k}{T}$). Pulling the vector out (recall that we use implicit summation), and using the definition of the $L^2$-product, this can be viewed as
\[
\underbrace{(\uvec{v}_T^j)^J}_{\text{vector}} \underbrace{\left(\int_T \Pcurl{k}\phi_i\cdot\Pcurl{k}\phi_j\,\Pgrad{k+1}\psi_k\right)\langle e_I,[e_J,e_K]\rangle}_{\text{trilinear form }M_{(i,I),(j,J),(k,K)}},
\]
where $M$ is indeed a trilinear form (indices $((i,I), (j,J), (k,K))$), represented by a 3-dimension block of $d^3\times(\dim(\Xcurl{k}{T}))\times(\dim(\Xcurl{k}{T}))\times(\dim(\Xgrad{k}{T}))$ coefficients. As mentioned, although $M$ would be useful to calculate in itself, since it can then be used to find other terms in the scheme, performing the contraction with each vector $\uvec{v}_T$ is actually quite slow because of the large matrix sums.

Instead we do something less intuitive, by combining the vector with the integral portion of the product first:
\begin{equation}\label{eq:trilinear.SDDR.La}
\underbrace{(\uvec{v}_T^j)^J\left(\int_T \Pcurl{k}\phi_i\cdot\Pcurl{k}\phi_j\,\Pgrad{k+1}\psi_k\right)}_{M^J_{i,k}}\underbrace{\langle e_I,[e_J,e_K]\rangle}_{N_{I,J,K}}.
\end{equation}
These integral coefficients represent the smaller SDDR trilinear form $\int_T(\Pcurl{k}\cdot)\cdot(\Pcurl{k}\cdot)(\Pgrad{k+1}\cdot)$; this term is first combined (through the sum over $j$) with the vector $(\uvec{v}_T)^J$, before performing the tensorisation with $N_{I,J,K}$ (on $(i,I),(k,K)$), and finally taking the much shorter matrix sum over $J$. This delay in combining the Lie algebra indices initially seems like extra work, as this $M^J_{i,k}$ sum is unique to each vector $(\uvec{v}_T)^J$, and so the tensorisation must be re-done each time, but testing showed that the tradeoff for smaller matrix sums is worth it in this case.

The $\ebkt{\cdot}{\cdot}$ bracket is dealt with differently because it can appear twice in a single product. In this double bracket case, if we use the same summation methods as in \eqref{eq:trilinear.SDDR.La}, then we have the appearance of quadrilinear forms instead of trilinear forms, that would require the computation of a 4-dimensional array. To avoid the extra dimension, we treat the bracket terms as independent objects, that can be manipulated separately to the product matrix. For example, in the first discretisation $\mathfrak N_1$ described in \eqref{disc.N.1}, two terms involve this bracket, which we compute the following way (underbrace denotes the dimension of the array representation with the implied transposition, and we write formal products to show how the calculation could be decomposed in the code):
\begin{align}\label{eq:double.bkt.1}
(\ebkttrk{\uvec{v}_h}{\cdot},\ebkttrk{\uvec{w}_h}{\cdot})_{\DIV,\La,h}={}&\underbrace{\ebkttrk{\uvec{v}_h}{\cdot}}_{\text{matrix}}\ \underbrace{(\cdot,\cdot)_{\DIV,\La,h}}_{\text{matrix}}\ \underbrace{\ebkttrk{\uvec{w}_h}{\cdot}}_{\text{matrix}}, \\ \label{eq:double.bkt.2}
(\ebkttrk{\uvec{v}_h}{\uvec{w}_h},\ebkttrk{\cdot}{\cdot})_{\DIV,\La,h}={}&\underbrace{\ebkttrk{\uvec{v}_h}{\uvec{w}_h}}_{\text{vector}}\ \underbrace{(\cdot,\cdot)_{\DIV,\La,h}}_{\text{matrix}}\ \underbrace{\ebkttrk{\cdot}{\cdot}}_{3\text{-dim array}}.
\end{align}
We note that the map $\ebkttrk{\cdot}{\cdot}:(\LaXcurl{k}{h})^2\to\LaXdiv{k}{h}$ is bilinear, but requires an extra dimension in the corresponding array to represent the output in $\LaXdiv{k}{h}$. Thus the terms $\ebkttrk{\uvec{v}_h}{\cdot}$ and $\ebkttrk{\uvec{w}_h}{\cdot}$ are indeed represented by matrices, calculated using the same principle (described in \eqref{eq:trilinear.SDDR.La}) of avoiding the tensorisation until the last step. The order of operations is also crucial in \eqref{eq:double.bkt.2}; the vector-matrix multiplication must come first, to ensure that the 3-dimensional array is only contracted with a vector. We stress again however, that the full set of $\ebkttrk{\cdot}{\cdot}$ coefficients is never constructed, and all calculations are done in the way of \eqref{eq:trilinear.SDDR.La}.

The same idea is implemented in the second discretisation $\mathfrak N_2$ (see \eqref{disc.N.2}), by introducing a basis for $\vPoly{2k}(T)$, and working with the map $\ebkt{\LaPcurl{k}\cdot}{\LaPcurl{k}\cdot}:(\LaXcurl{k}{T})^2\to\vPoly{2k}(T)\otimes\La$. The numerical decomposition of the term analogous to \eqref{eq:double.bkt.1} is
\begin{equation}\label{eq:double.bkt.N2}
\underbrace{\ebkt{\LaPcurl{k}\uvec{v}}{\LaPcurl{k}\cdot}}_{\text{matrix}}\underbrace{\left(\int_T\langle\cdot,\cdot\rangle\right)}_{\text{matrix}}\underbrace{\ebkt{\LaPcurl{k}\uvec{w}}{\LaPcurl{k}\cdot}}_{\text{matrix}},
\end{equation}
where the integral is realised by the tensorisation of the mass matrix of the basis on $\vPoly{2k}(T)$, and the mass matrix of the Lie algebra (see \eqref{matrix.tensor}). With orthonormal choices of bases of $\vPoly{2k}(T)$ and $\mathfrak g$ (which is the default in the \texttt{HArDCore} library), the calculations here are greatly simplified; the components of $\ebkt{\LaPcurl{k}\cdot}{\LaPcurl{k}\cdot}$ can be found using only the triple integrals of the bases of $\Poly{k}(T)$ and $\vPoly{2k}(T)$ (without requiring to solve a linear system afterwards), and the integral in \eqref{eq:double.bkt.N2} is just given by the identity matrix $\Id{d(2k+1)}=\Id{2k+1}\otimes\Id{d}$.

%% Tests
\section{Numerical tests}\label{sec:tests}

We present a numerical comparison of the convergence of the two schemes, as well as the exact constraint preservation that is expected.  The tests were performed on a Dell Precision 5820 desktop with a 14-core Intel Xeon processor (W-2275) clocked at 3.3 GHz and equipped with 128 GB of DDR4 RAM, running Ubuntu 22.04.1 LTS.
 The discretisation setting is identical to that of \cite[Section 5]{Droniou.Oliynyk.ea:23}, which only contained tests for $k=0$ of the scheme pertaining to $\mathfrak N_1$. Here we expand on these results for higher orders ($k=0,1,2$), and also consider the performance in relation to the second discretisation. 

Let us recall the setting of these tests. The Lie algebra is $\La=\mathfrak{su}(2)$, with basis 
\[
e_1=-\frac{i}{2}\left[\begin{matrix} 0&1\\ 1&0\end{matrix}\right]\,,\quad
e_2=-\frac{i}{2}\left[\begin{matrix} 0&-i\\ i&0\end{matrix}\right]\,,\quad
e_3=-\frac{i}{2}\left[\begin{matrix} 1&0\\ 0&-1\end{matrix}\right].
\]
The time interval is $[0,1]$ and the space domain is the unit cube $(0,1)^3$; the spatial discretisation is based on three families of Voronoi, tetrahedral, and cubic cell meshes. For each mesh of size $h$, the time interval is uniformly divided into $\max\{10,\left \lceil{5/h}^{k+1}\right \rceil\}$ time steps; given that we use an implicit time discretisation, the expected rate of convergence is in $\deltat+h^{k+1}$, and the choice of time step is designed so that, when plotted against $h$, the errors should decay as $h^{k+1}$. To set non-zero initial conditions, and to assess the convergence properties of the schemes, we select a manufactured solution based on
\begin{equation}\label{eq:gauge.A}
\begin{aligned}
  \bvec{A}(t)={}&
  \begin{bmatrix}
    -0.5\cos(t)\sin(\pi x)\cos(\pi y)\cos(\pi z)\\ 
    \cos(t)\cos(\pi x)\sin(\pi y)\cos(\pi z)\\ 
    -0.5\cos(t)\cos(\pi x)\cos(\pi y)\sin(\pi z)
  \end{bmatrix}
  \otimes e_1 +
  \begin{bmatrix}
    -0.5\sin(t)\sin(\pi x)\cos(\pi y)\cos(\pi z)\\ 
    \sin(t)\cos(\pi x)\sin(\pi y)\cos(\pi z)\\ 
    -0.5\sin(t)\cos(\pi x)\cos(\pi y)\sin(\pi z)
  \end{bmatrix}
  \otimes e_2 \\
  &+
  \begin{bmatrix}
    -0.5\sin(t)\sin^2(\pi y)\\ 
    \cos(t)\cos^2(\pi z)\\ 
    -0.5\sin(t)\cos^2(\pi x)
  \end{bmatrix}\otimes e_3,
\end{aligned}
\end{equation}
from which $\bvec{E}(t)$ is calculated using \eqref{eq:lm.ym.0}. Then for all tests in this section, the initial conditions are assumed to be $\uvec{A}_h^0=\LaIcurl{k}{h}\bvec{A}(0),\uvec{E}_h^0=\LaIcurl{k}{h}\bvec{E}(0)$.

\begin{remark}[Convergence results for $\underline{\lambda}_h$]
Although the fields $\bvec{A}$, $\bvec{E}$ of a solution to the constrained formulation \eqref{eq:lm.ym} solve the Yang--Mills equations, there is no proof that the $\lambda$ obtained is unique. Testing performed in \cite[Section 5]{Droniou.Oliynyk.ea:23} suggest indeed that there are infinitely many solutions; the values obtained for $\underline{\lambda}_h^n$ are therefore not very instructive, and have been omitted from the graphs. Discussion around the solvability of the linear system deriving from such a scheme can also be found in the cited section; we experienced a similar success, with a worst residual for the linear solver of the order $1$e$-09$.
\end{remark}

\subsection{Convergence tests}

In addition, appropriate boundary conditions and forcing terms are introduced to balance the equations (see \cite[Section 5.1]{Droniou.Oliynyk.ea:23} for details). The errors for both schemes are measured by calculating the difference $\norm{\LaXcurl{k}{h}}{\uvec{E}_h^{n+1}-\LaIcurl{k}{h}\bvec{E}(1)}$ (resp. $\uvec{A}_h^{n+1},\bvec{A}$) at the final time, and dividing by the norm of $\uvec{E}_h^{n+1}$ (resp. $\uvec{A}_h^{n+1}$). These relative errors are plotted in Figure \ref{fig:elec.errors} for $\bvec{E}$, and Figure \ref{fig:pot.errors} for $\bvec{A}$.

We remark immediately that for $\bvec{A}$, the errors are indistinguishable from the figure alone. The precise difference for the Voronoi and cubic sequences are calculated in Table \ref{tab:diff.err.A}, with the trend applying identically to the tetrahedral family. As either $k$ increases or $h$ decreases, the difference between the errors shrink accordingly, and this is seen also in the errors for $\bvec{E}$ for $k=1,2$. However for the electric field, there is a more visible separation when $k=0$, with the $\mathfrak N_1$-based discretisation performing slightly better. These tests seem to indicate that, overall, both choices $\mathfrak N_1,\mathfrak N_2$ lead to acceptable and similar results.

The schemes converge on the Voronoi and cubic meshes at the expected rate of $k+1$ for both $\bvec{E}$ and $\bvec{A}$, but this behaviour was not as stable for $\bvec{E}$ on the tetrahedral line, where we see a rate that jumps around $3$ for every $k$. This might be due to the asymptotic regime not been reached yet on these meshes (we note that, for $k=0$ for example, the simulation on the finest mesh seem to indicate that the convergence rate slows down). The magnitude of these errors are still ordered in the expected way, except for the coarsest cubic mesh in Figure \ref{fig:elec.errors}, where it is corrected after the first refinement.

%%% Some errors %%%
\begin{figure}\centering
  \ref{elec.fields}
  \vspace{0.50cm}\\
  \begin{minipage}{0.45\textwidth}
    \begin{tikzpicture}[scale=0.85]
      \begin{loglogaxis} [legend columns=3, legend to name=elec.fields]  
        \logLogSlopeTriangle{0.90}{0.4}{0.1}{1}{black};
        \logLogSlopeTriangle{0.90}{0.4}{0.1}{2}{black};
        \logLogSlopeTriangle{0.90}{0.4}{0.1}{3}{black};
        \addplot [mark=star, red] table[x=MeshSize,y=E_L2Elec] {outputs/N1/Voro-small-0_k0/data_rates.dat};
        \addlegendentry{$\bvec{E}$ ($\mathfrak N_1$, $k=0$);}
        \addplot [mark=star, mark options=solid, red, dashed] table[x=MeshSize,y=E_L2Elec] {outputs/N1/Voro-small-0_k1/data_rates.dat};
        \addlegendentry{$\bvec{E}$ ($\mathfrak N_1$, $k=1$);}
        \addplot [mark=star, mark options=solid, red, dotted] table[x=MeshSize,y=E_L2Elec] {outputs/N1/Voro-small-0_k2/data_rates.dat};
        \addlegendentry{$\bvec{E}$ ($\mathfrak N_1$,  $k=2$);}
        \addplot [mark=*, blue] table[x=MeshSize,y=E_L2Elec] {outputs/N2/Voro-small-0_k0/data_rates.dat};
        \addlegendentry{$\bvec{E}$ ($\mathfrak N_2$, $k=0$);}
        \addplot [mark=*, mark options=solid, blue, dashed] table[x=MeshSize,y=E_L2Elec] {outputs/N2/Voro-small-0_k1/data_rates.dat};
        \addlegendentry{$\bvec{E}$ ($\mathfrak N_2$, $k=1$);}
        \addplot [mark=*, mark options=solid, blue, dotted] table[x=MeshSize,y=E_L2Elec] {outputs/N2/Voro-small-0_k2/data_rates.dat};
      	\addlegendentry{$\bvec{E}$ ($\mathfrak N_2$, $k=2$);}
      \end{loglogaxis}          
    \end{tikzpicture}
    \subcaption{``Voro-small-0'' mesh}
  \end{minipage}
  \begin{minipage}{0.45\textwidth}
    \begin{tikzpicture}[scale=0.85] 
      \begin{loglogaxis}
        \logLogSlopeTriangle{0.90}{0.4}{0.1}{1}{black};
        \logLogSlopeTriangle{0.90}{0.4}{0.1}{2}{black};
        \logLogSlopeTriangle{0.90}{0.4}{0.1}{3}{black};
        \addplot [mark=star, red] table[x=MeshSize,y=E_L2Elec] {outputs/N1/Tetgen-Cube-0_k0/data_rates.dat};
        \addplot [mark=*, blue] table[x=MeshSize,y=E_L2Elec] {outputs/N2/Tetgen-Cube-0_k0/data_rates.dat};
        \addplot [mark=star, mark options=solid, red, dashed] table[x=MeshSize,y=E_L2Elec] {outputs/N1/Tetgen-Cube-0_k1/data_rates.dat};
        \addplot [mark=*, mark options=solid, blue, dashed] table[x=MeshSize,y=E_L2Elec] {outputs/N2/Tetgen-Cube-0_k1/data_rates.dat};
        \addplot [mark=star, mark options=solid, red, dotted] table[x=MeshSize,y=E_L2Elec] {outputs/N1/Tetgen-Cube-0_k2/data_rates.dat};
        \addplot [mark=*, mark options=solid, blue, dotted] table[x=MeshSize,y=E_L2Elec] {outputs/N2/Tetgen-Cube-0_k2/data_rates.dat};         
        \end{loglogaxis} 
      \end{tikzpicture}
    \subcaption{``Tetgen-Cube-0'' mesh}
  \end{minipage}\\[0.5em]
  \begin{minipage}{0.45\textwidth}
    \begin{tikzpicture}[scale=0.85] 
      \begin{loglogaxis}
        \logLogSlopeTriangle{0.90}{0.4}{0.1}{1}{black};
        \logLogSlopeTriangle{0.90}{0.4}{0.1}{2}{black};
        \logLogSlopeTriangle{0.90}{0.4}{0.1}{3}{black};
        \addplot [mark=star, red] table[x=MeshSize,y=E_L2Elec] {outputs/N1/Cubic-Cells_k0/data_rates.dat};
        \addplot [mark=*, blue] table[x=MeshSize,y=E_L2Elec] {outputs/N2/Cubic-Cells_k0/data_rates.dat};
        \addplot [mark=star, mark options=solid, red, dashed] table[x=MeshSize,y=E_L2Elec] {outputs/N1/Cubic-Cells_k1/data_rates.dat};
        \addplot [mark=*, mark options=solid, blue, dashed] table[x=MeshSize,y=E_L2Elec] {outputs/N2/Cubic-Cells_k1/data_rates.dat}; 
        \addplot [mark=star, mark options=solid, red, dotted] table[x=MeshSize,y=E_L2Elec] {outputs/N1/Cubic-Cells_k2/data_rates.dat};
        \addplot [mark=*, mark options=solid, blue, dotted] table[x=MeshSize,y=E_L2Elec] {outputs/N2/Cubic-Cells_k2/data_rates.dat};           
        \end{loglogaxis} 
      \end{tikzpicture}
    \subcaption{``Cubic-Cells'' mesh}
  \end{minipage}\\[0.5em]
  \caption{Relative errors on $\bvec{E}$ for $\mathfrak N_1$ and $\mathfrak N_2$: Voronoi, tetrahedral and cubic meshes}
  \label{fig:elec.errors}
\end{figure}
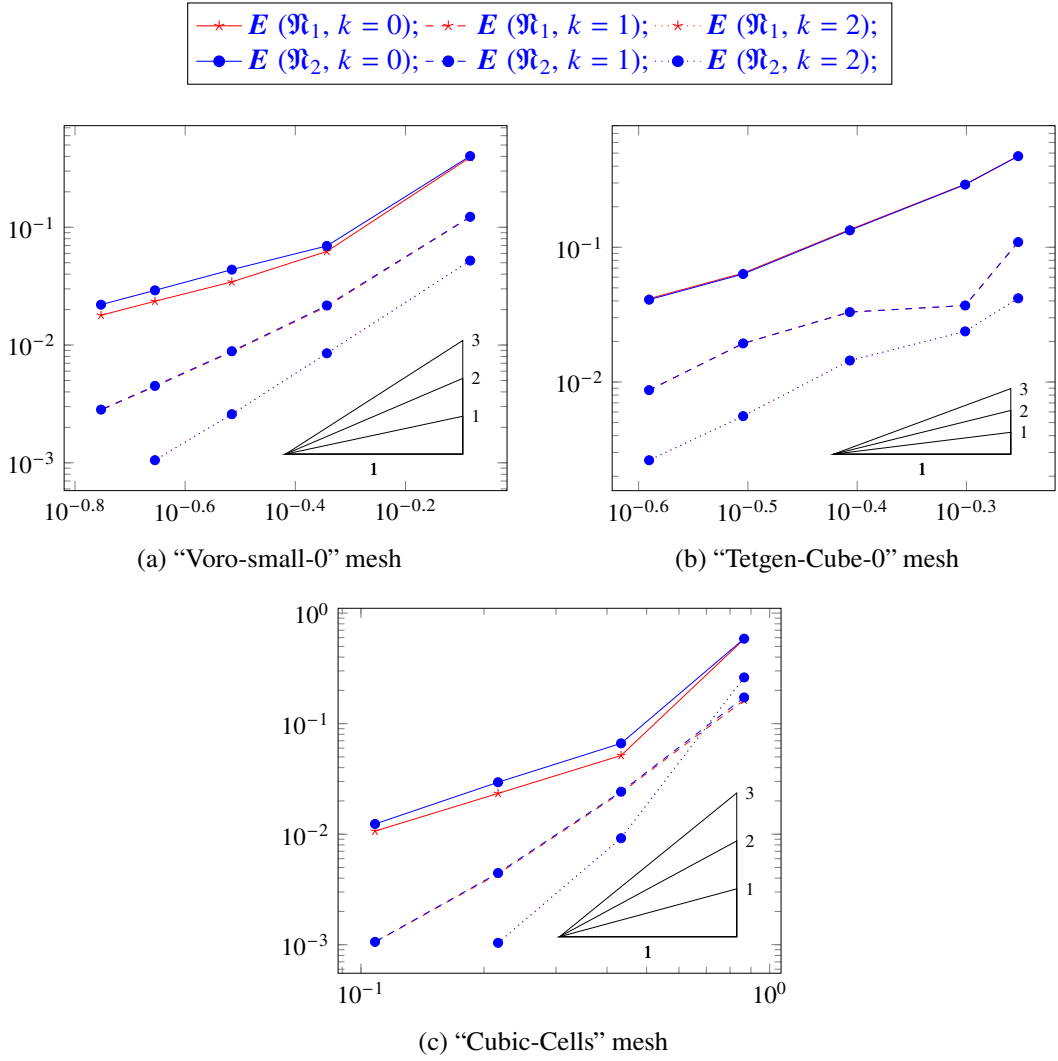

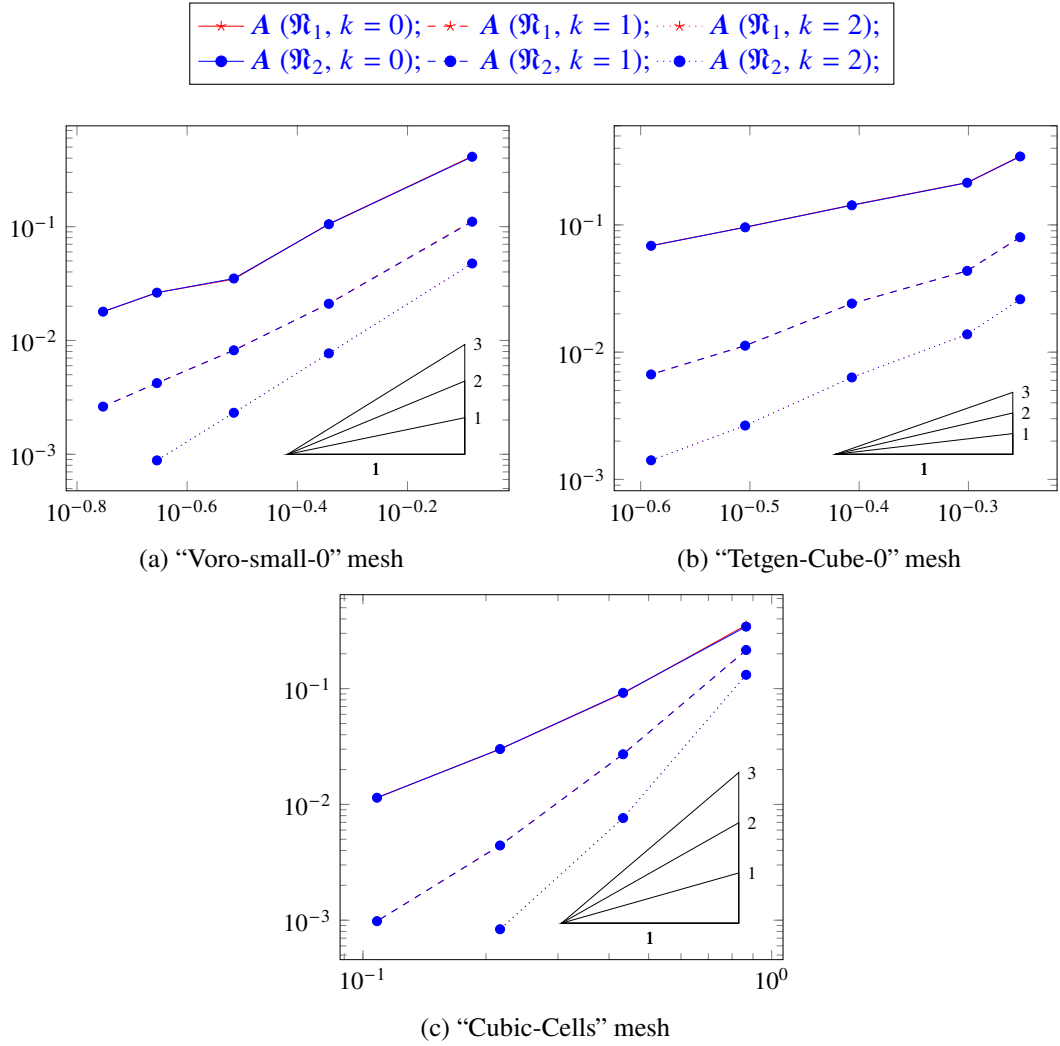
\begin{figure}\centering
  \ref{pot.fields}
  \vspace{0.50cm}\\
  \begin{minipage}{0.45\textwidth}
    \begin{tikzpicture}[scale=0.85]
      \begin{loglogaxis} [legend columns=3, legend to name=pot.fields]  
        \logLogSlopeTriangle{0.90}{0.4}{0.1}{1}{black};
        \logLogSlopeTriangle{0.90}{0.4}{0.1}{2}{black};
        \logLogSlopeTriangle{0.90}{0.4}{0.1}{3}{black};
        \addplot [mark=star, red] table[x=MeshSize,y=E_L2Pot] {outputs/N1/Voro-small-0_k0/data_rates.dat};
        \addlegendentry{$\bvec{A}$ ($\mathfrak N_1$, $k=0$);}
        \addplot [mark=star, mark options=solid, red, dashed] table[x=MeshSize,y=E_L2Pot] {outputs/N1/Voro-small-0_k1/data_rates.dat};
        \addlegendentry{$\bvec{A}$ ($\mathfrak N_1$, $k=1$);}
        \addplot [mark=star, mark options=solid, red, dotted] table[x=MeshSize,y=E_L2Pot] {outputs/N1/Voro-small-0_k2/data_rates.dat};
        \addlegendentry{$\bvec{A}$ ($\mathfrak N_1$,  $k=2$);}
        \addplot [mark=*, blue] table[x=MeshSize,y=E_L2Pot] {outputs/N2/Voro-small-0_k0/data_rates.dat};
        \addlegendentry{$\bvec{A}$ ($\mathfrak N_2$, $k=0$);}
        \addplot [mark=*, mark options=solid, blue, dashed] table[x=MeshSize,y=E_L2Pot] {outputs/N2/Voro-small-0_k1/data_rates.dat};
        \addlegendentry{$\bvec{A}$ ($\mathfrak N_2$, $k=1$);}
        \addplot [mark=*, mark options=solid, blue, dotted] table[x=MeshSize,y=E_L2Pot] {outputs/N2/Voro-small-0_k2/data_rates.dat};
      	\addlegendentry{$\bvec{A}$ ($\mathfrak N_2$, $k=2$);}
      \end{loglogaxis}          
    \end{tikzpicture}
    \subcaption{``Voro-small-0'' mesh}
  \end{minipage}
  \begin{minipage}{0.45\textwidth}
    \begin{tikzpicture}[scale=0.85] 
      \begin{loglogaxis}
        \logLogSlopeTriangle{0.90}{0.4}{0.1}{1}{black};
        \logLogSlopeTriangle{0.90}{0.4}{0.1}{2}{black};
        \logLogSlopeTriangle{0.90}{0.4}{0.1}{3}{black};
        \addplot [mark=star, red] table[x=MeshSize,y=E_L2Pot] {outputs/N1/Tetgen-Cube-0_k0/data_rates.dat};
        \addplot [mark=*, blue] table[x=MeshSize,y=E_L2Pot] {outputs/N2/Tetgen-Cube-0_k0/data_rates.dat};
        \addplot [mark=star, mark options=solid, red, dashed] table[x=MeshSize,y=E_L2Pot] {outputs/N1/Tetgen-Cube-0_k1/data_rates.dat};
        \addplot [mark=*, mark options=solid, blue, dashed] table[x=MeshSize,y=E_L2Pot] {outputs/N2/Tetgen-Cube-0_k1/data_rates.dat};
        \addplot [mark=star, mark options=solid, red, dotted] table[x=MeshSize,y=E_L2Pot] {outputs/N1/Tetgen-Cube-0_k2/data_rates.dat};
        \addplot [mark=*, mark options=solid, blue, dotted] table[x=MeshSize,y=E_L2Pot] {outputs/N2/Tetgen-Cube-0_k2/data_rates.dat};         
        \end{loglogaxis} 
      \end{tikzpicture}
    \subcaption{``Tetgen-Cube-0'' mesh}
  \end{minipage}\\[0.5em]
  \begin{minipage}{0.45\textwidth}
    \begin{tikzpicture}[scale=0.85] 
      \begin{loglogaxis}
        \logLogSlopeTriangle{0.90}{0.4}{0.1}{1}{black};
        \logLogSlopeTriangle{0.90}{0.4}{0.1}{2}{black};
        \logLogSlopeTriangle{0.90}{0.4}{0.1}{3}{black};
        \addplot [mark=star, red] table[x=MeshSize,y=E_L2Pot] {outputs/N1/Cubic-Cells_k0/data_rates.dat};
        \addplot [mark=*, blue] table[x=MeshSize,y=E_L2Pot] {outputs/N2/Cubic-Cells_k0/data_rates.dat};
        \addplot [mark=star, mark options=solid, red, dashed] table[x=MeshSize,y=E_L2Pot] {outputs/N1/Cubic-Cells_k1/data_rates.dat};
        \addplot [mark=*, mark options=solid, blue, dashed] table[x=MeshSize,y=E_L2Pot] {outputs/N2/Cubic-Cells_k1/data_rates.dat}; 
        \addplot [mark=star, mark options=solid, red, dotted] table[x=MeshSize,y=E_L2Pot] {outputs/N1/Cubic-Cells_k2/data_rates.dat};
        \addplot [mark=*, mark options=solid, blue, dotted] table[x=MeshSize,y=E_L2Pot] {outputs/N2/Cubic-Cells_k2/data_rates.dat};           
        \end{loglogaxis}  
      \end{tikzpicture}
    \subcaption{``Cubic-Cells'' mesh}
  \end{minipage}\\[0.5em]
  \caption{Relative errors on $\bvec{A}$ for $\mathfrak N_1$ and $\mathfrak N_2$: Voronoi, tetrahedral and cubic meshes}
  \label{fig:pot.errors}
\end{figure}

\begin{table}\centering
\begin{tabular}{c|ccccc|cccc|}
 & \multicolumn{5}{c|}{Voronoi mesh} & \multicolumn{4}{c|}{Cubic mesh}\\
 & 1 & 2 & 3 & 4 & 5 & 1 & 2 & 3 & 4 \\
 \hline
$k=0$ & 4.36e-3 & 5.98e-4 & 5.02e-4 & 1.67e-5 & 1.43e-5 & 1.05e-2 & 9.97e-4 & 8.98e-5 & 1.11e-5\\
$k=1$ & 1.41e-3  & 8.93e-5 & 1.47e-5 & 6.23e-6 & 2.57e-6 & 9.99e-4 & 1.05e-4 & 7.04e-6 & 4.35e-7 \\
$k=2$ & 9.16e-5 & 3.46e-6 & 3.4e-7 & 6e-8 & - & 1.01e-4 & 4.97e-6 & 1.96e-7  & -
\end{tabular}
\caption{Difference between the errors of $\bvec{A}$ ($\mathfrak N_1$) and $\bvec{A}$ ($\mathfrak N_2$) for various degrees $k$ on the Voronoi and Cubic meshes}
\label{tab:diff.err.A}
\end{table}

\subsection{Constraint preservation}\label{sec:tests.conv}

The tests for the preservation of constraint are run with the same initial conditions as the convergence tests. For proper solutions to the Yang--Mills equations, and for Proposition \ref{prop:energy.ym}, it is expected that these discrete fields prescribe a small or vanishing initial constraint $\mathfrak C^0$. This can be achieved by projecting the initial conditions (see \cite{Droniou.Oliynyk.ea:23}), however for the purpose of testing Proposition \ref{prop:constraint.preservation}, which does not depend on the initial values, it suffices to measure the maximum change in the functional $\mathfrak C^n-\mathfrak C^0$ over all times. The differences are presented in Table \ref{tab:constr.preserv} for selected meshes from each sequence. We see for both methods that the constraint is stationary up to machine precision, with the small drift coming from rounding errors present at each iteration.

For reference, we also provide in Table \ref{tab:runtimes} a comparison of the runtimes for the two different choices of discretisation of the nonlinear tests. The performances of both schemes are very comparable on a variety of meshes and degrees $k$, with the exception of a $10\%$ difference in favour of $\mathfrak N_1$ for $k=2$ on the Tetrahedral meshes, that shows up consistently through our tests. This difference is likely due to the nontrivial calculation attached to each face component \eqref{eq:def.ebkttrk.defF} in the definition of the discrete bracket of $\mathfrak N_1$. These calculations are necessary because the face values also contribute to the $L^2$-product, but they result in a larger dependency of the runtime on the number of faces in a particular mesh. In comparison, the $\mathfrak N_2$ discretisation is less affected by the shape of the elements; numerically, an extra face only represents an increase in size of the SDDR operators (which equally affects $\mathfrak N_1$) used in the integrals. This is supported by the results for the Voronoi sequence, that has a higher face to element ratio than the tetrahedral sequence, where $\mathfrak N_2$ starts to slightly outperform $\mathfrak N_1$.

\begin{table}\centering
\begin{tabular}{c|cc|cc|cc|}
 & \multicolumn{2}{c|}{Voronoi mesh} & \multicolumn{2}{c|}{Tetrahedral mesh} & \multicolumn{2}{c|}{Cubic mesh}\\
 $\mathfrak N_1$ & 1 & 3 & 2 & 4 & 1 & 3 \\
 \hline
$k=0$ & 8.47329e-15 & 3.05676e-14 & 2.06362e-14 & 4.75412e-14 & 3.52318e-15 & 2.16527e-14\\
$k=1$ & 1.4144e-13  & 8.93075e-13 & 3.67781e-13 & 1.81426e-12 & 2.33678e-14 & 6.44019e-13\\
$k=2$ & 3.69918e-12 & 1.18207e-10 & 3.82037e-12 & 2.77407e-11 & 4.45312e-14 & 6.48608e-12\\
 \hline
$\mathfrak N_2$ & 1 & 3 & 2 & 4 & 1 & 3 \\
 \hline
$k=0$ & 8.16124e-15 & 3.13617e-14 & 2.01667e-14 & 4.77633e-14 & 4.22851e-15 & 2.11083e-14\\
$k=1$ & 9.8531e-14  & 8.83056e-13 & 3.69107e-13 & 1.81787e-12 & 2.52977e-14 & 6.48934e-13\\
$k=2$ & 4.16428e-12 & 7.99416e-11 & 3.81537e-12 & 2.77391e-11 & 4.51672e-14 & 6.48285e-12
\end{tabular}
\caption{Maximum over $n$ of the difference $\mathfrak C^n-\mathfrak C^0$ measured in the dual norm}
\label{tab:constr.preserv}
\end{table}

\begin{table}\centering
\begin{tabular}{c|cc|cc|cc|}
 & \multicolumn{2}{c|}{Voronoi mesh} & \multicolumn{2}{c|}{Tetrahedral mesh} & \multicolumn{2}{c|}{Cubic mesh}\\
 $\mathfrak N_1$ & 1 & 3 & 2 & 4 & 1 & 3 \\
 \hline
$k=0$ & 5.00865 & 145.77  & 4.70017 & 21.1378 & 0.564665 & 35.2541\\
$k=1$ & 35.9231 & 2836.36 & 50.997  & 360.943 & 3.58296  & 588.679\\
$k=2$ & 198.435 & 43303.9 & 578.499 & 5732.1  & 14.6638  & 14337.4\\
 \hline
$\mathfrak N_2$ & 1 & 3 & 2 & 4 & 1 & 3 \\
 \hline
$k=0$ & 4.57515 & 135.231 & 4.16998 & 19.7708 & 0.53204 & 32.879\\
$k=1$ & 34.2036 & 2814.14 & 51.3249 & 340.817 & 3.62877 & 631.546\\
$k=2$ & 190.447 & 42083.9 & 634.162 & 6421.87 & 13.8524 & 14414.9
\end{tabular}
\caption{Total runtime for each test in seconds}
\label{tab:runtimes}
\end{table}

%% Conclusion
\section{Conclusion}\label{sec:conclusion}

We designed two schemes for the Yang--Mills equations based on the Discrete de Rham method, both displaying arbitrary orders of accuracy and applications on generic polyhedral meshes. Thanks to the complex property of DDR and to the usage of a Lagrange multiplier, both schemes also preserve a discrete nonlinear constraint deriving from the Yang--Mills equations, and satisfy energy bounds. The schemes only differ in their treatment of the nonlinearity akin to a cross product combined with the Lie bracket for Lie algebra-valued vector functions. The first scheme reconstructs a discrete version of the continuous product bracket, that can then be used in the discrete $L^2$-products of the DDR complex. The second scheme uses the DDR potential reconstructions to get polynomials in each element, on which the continuous product bracket can be applied. We show how a clever ordering of the algebraic operations in the assembly of the schemes can help keep the computational cost at a reasonable level, despite needing to deal with multidimensional arrays and high system sizes due to the Lie algebra components. Numerical results are presented which show a good behaviour and an expected rate of convergence, with respect to the mesh size, in $h^{k+1}$.

\section*{Acknowledgements}
Funded by the European Union (ERC, NEMESIS, No.~101115663).
Views and opinions expressed are however those of the authors only and do not necessarily reflect those of the European Union or the European Research Council Executive Agency. Neither the European Union nor the granting authority can be held responsible for them.

%------------------------------------------------------------------------------%
% Bibliography
%------------------------------------------------------------------------------%

\printbibliography

\end{document}